\documentclass[11pt,reqno]{amsart}

\usepackage{lineno,hyperref}

\modulolinenumbers[5]

\usepackage{amsmath}
\usepackage{amssymb}
\usepackage{mathrsfs}
\usepackage{amsthm}
\usepackage{bm}
\usepackage{graphicx}
\usepackage{subfigure}
\usepackage{xcolor}  
\usepackage{tikz} 
\usetikzlibrary{arrows,shapes,chains}  
\usepackage{booktabs}
\usepackage{float}
\usepackage{geometry}
\usepackage{color}

\newtheorem{Def}{Definition}[section]
\newtheorem{lem}[Def]{Lemma}
\newtheorem{theo}[Def]{Theorem}
\newtheorem{pro}[Def]{Proposition}
\newtheorem{rem}[Def]{Remark}

\newtheorem{assum}{Assumption}

\usepackage{pgf}
\definecolor{Green}{RGB}{0,128,0}

\newcommand{\LL}{\langle}
\newcommand{\RR}{\rangle}

\newcommand{\mcal}{\mathcal}

\newcommand{\mscr}{\mathscr}
\newcommand{\mbb}{\mathbb}
\newcommand{\mbf}{\mathbf}

\newcommand{\ud}{\mathrm d}

\newcommand{\PD}{\partial}
\numberwithin{equation}{section}
\allowdisplaybreaks \allowdisplaybreaks[4]%change line

\begin{document}

\title[CLT]{Central limit theorem for temporal average of backward Euler--Maruyama method}

\author{Diancong Jin}
\address{School of Mathematics and Statistics, Huazhong University of Science and Technology, Wuhan 430074, China;
	Hubei Key Laboratory of Engineering Modeling and Scientific Computing, Huazhong University of Science and Technology, Wuhan 430074, China}
\email{jindc@hust.edu.cn}

\thanks{This work is supported by National Natural Science Foundation of China (No. 12201228), and the Fundamental Research Funds for the Central Universities 3004011142.
}


\begin{abstract}
This work focuses on the temporal average  of the backward  Euler--Maruyama (BEM) method, which is used to approximate the ergodic limit of  stochastic ordinary differential equations with super-linearly growing drift coefficients. 
We give the central limit theorem (CLT) of the  temporal average, which characterizes the asymptotics in distribution of the temporal average. When the deviation order is smaller than the optimal strong order, we directly derive the CLT of the temporal average through that of  original equations and the uniform strong order of the BEM method. For the case that the deviation order equals to the optimal strong order, the CLT is established via the Poisson equation associated with the generator of original equations. Numerical experiments are  performed to illustrate the theoretical results.

\end{abstract}

%\end{frontmatter}
%\linenumbers
\maketitle

%60F05 (1973-now) Central limit and other weak theorems
%60H10 (1973-now) Stochastic ordinary differential equations (aspects of stochastic analysis) [See also 34F05]
%60H35 (2000-now) Computational methods for stochastic equations (aspects of stochastic analysis)
%65C30 (2000-now) Numerical solutions to stochastic differential and integral equations {For theoretical aspects, see 60H35} [See also 65M75, 65N75]

\section{Introduction}
Ergodic theory is a powerful tool to investigate the long-time dynamics and 
statistical properties of stochastic systems, which is widely applied in physics, biology, chemistry and so on(see, e.g., \cite{ergodic1982,Wangxu,ergodic2015,ergodic2016}). A crucial problem in ergodic theory is to determine the ergodic measure and ergodic limit. Since  explicit expressions of them are generally  unavailable, one usually resorts to numerical methods to obtain their  approximations. There have been lots of numerical methods  which inherit the ergodicity or approximate the ergodic limit of original systems; see \cite{Vilmart14,HSW17,HWZ17,Mattingly10,Pardoux05,Talay1984} and references therein. In the aforementioned work, main efforts are made to analyze the error between the numerical
invariant measure and the original one, and that between numerical temporal average and the ergodic limit. 

Besides the convergence of the numerical temporal average in the moment sense, the asymptotics of its distribution is also an essential property. In recent several work, the  central limit theorem (CLT) of the  temporal average of some numerical methods is given, which characterizes the fluctuation of the numerical temporal average around  ergodic limits of original systems in the sense of distribution. In \cite{CLT2012}, the CLT of the temporal average of the Euler-Maruyama (EM) method with decreasing step-size for  ergodic stochastic ordinary differential equations (SODEs) is given. In addition, \cite{Xulihu22} proves the CLT and moderate deviation of the EM method with a fixed step-size for SODEs. For a class of ergodic stochastic partial differential equations (SPDEs), \cite{Dang23} shows that the temporal average of a full discretization with fixed temporal and spatial step-sizes satisfies the CLT. In the existing work, the CLT of numerical temporal average is established  provided that coefficients of original equations are Lipschitz continuous. This motivates us to investigate  the CLT of numerical temporal average for stochastic systems with non-Lipschitz coefficients, which have more extensive applications in reality compared with the Lipschitz case.

In this work, we consider the following SODE
\begin{align}\label{SDE}
		\ud X(t)=b(X(t))\ud t+\sigma(X(t))\ud W(t),~t>0,	
\end{align}
where $\{W(t),t\ge 0\}$ is a $D$-dimensional standard Brownian motion defined on a complete filtered probability space $\left(\Omega,\mcal F,\{\mcal F_t\}_{t\ge0},\mbf P\right)$, and $b:\mbb R^d\to\mbb R^d$ and $\sigma:\mbb R^d\to\mbb R^{d\times D}$ satisfy Assumptions \ref{assum1}-\ref{assum2} such that \eqref{SDE} admits a unique strong solution on $[0,+\infty)$ for any given deterministic initial value $X(0)\in\mbb R^d$. Notice that our assumptions allow $b$ to grow super-linearly.  It is shown in  \cite[Theorem 3.1]{Liuwei23} that \eqref{SDE} admits a unique invariant measure $\pi$ and is thus ergodic, due to the strong dissipation condition on $b$. In order to inherit the ergodicity of \eqref{SDE} and approximate the ergodic limit $\pi(h):=\int_{\mbb R^d}h(x)\pi(\ud x)$, $h\in\mbf C_b(\mbb R^d)$, \cite{Liuwei23} discretizes \eqref{SDE} by the backward Euler--Maruyama (BEM) method (see \eqref{BEM}), and gives the error between the numerical invariant measure $\pi_\tau$ and $\pi$ with $\tau$ being the step-size. The above result together with the strong order of the BEM method in the infinite time horizon implies that the temporal average $\frac{1}{N}\sum_{k=0}^Nh(\bar{X}_k^x)$ converges to the ergodic limit $\pi(h)$, i.e., 
\begin{gather}\label{Eq13}
	\lim_{\tau\to0}\lim_{N\to+\infty} \Big|\frac{1}{N}\sum_{k=0}^N\mbf Eh(\bar{X}_k^x)-\pi(h)\Big|=0,
\end{gather}
where $\{\bar{X}^x_n\}_{n\ge 0}$ is the numerical solution generated by the BEM method with initial value $x\in\mbb R^d$. 

The purpose of this paper is to present the CLT for the following temporal average 
\begin{align*}
	\Pi_{\tau,\alpha}(h)=\frac{1}{\tau^{-\alpha}}\sum_{k=0}^{\tau^{-\alpha}-1}h(\bar{X}^x_k), \quad\alpha\in(1,2],~h\in \mbf C^4_b(\mbb R^d),
\end{align*}
where for convenience we always assume that $\tau^{-\alpha}$ is an integer instead of the step number $N$ in \eqref{Eq13}. More precisely, we prove in Theorems \ref{CLT1} and \ref{CLT2}  that the normalized temporal average $\frac{1}{\tau^{\frac{\alpha-1}{2}}}(\Pi_{\tau,\alpha}(h)-\pi(h))$ converges to the normal distribution $\mcal N(0,\pi(|\sigma^\top\nabla \varphi|^2))$ in distribution as $\tau\to 0$, respectively for $\alpha\in(1,2)$ and $\alpha=2$. In fact, Theorem \ref{CLT1} indicates that the CLT holds for the temporal average of a class of numerical methods with uniform strong order $\frac{1}{2}$, for $\alpha\in(1,2)$.  
Here,  $\varphi$ is defined by \eqref{eq10} and solves the Poisson equation $\mcal L\varphi=h-\pi(h)$ (see Lemma \ref{varphi}), with $\mcal L$ being the generator of \eqref{SDE}. We call the  parameter $\tau^{\frac{\alpha-1}{2}}$ the deviation scale and $\frac{\alpha-1}{2}$ the deviation order; see Remark \ref{rem1} for the reason of requiring $\alpha>1$.

The proof ideas of the CLT for $\Pi_{\tau,\alpha}(h)$ are different for $\alpha\in(1,2)$ and $\alpha=2$.  For the case $\alpha\in(1,2)$,  we directly derive the CLT  for $\Pi_{\tau,\alpha}(h)$ in Theorem \ref{CLT1}, by means of the CLT for \eqref{SDE} and the optimal strong order in the infinite time horizon of the BEM method, considering that the CLT for \eqref{SDE} is a classical result (see \cite[Theorem 2.1]{FCLT82}. The key of this proof lies in that the deviation order $\frac{\alpha-1}{2}$ is smaller than the optimal strong order $\frac12$ for $\alpha\in(1,2)$, which  does not apply to the case $\alpha=2$.  In order to tackle the  more subtle case $\alpha=2$, we follow the argument in \cite{Xulihu22} and \cite{Dang23} to obtain the CLT for $\Pi_{\tau,2}(h)$. The main idea is to reformulate the normalized temporal average $\frac{1}{\tau^{\frac{\alpha-1}{2}}}(\Pi_{\tau,\alpha}(h)-\pi(h))$ by means of the Poisson  equation. This allows us to decompose   $\frac{1}{\tau^{\frac{\alpha-1}{2}}}(\Pi_{\tau,\alpha}(h)-\pi(h))$ as a martingale difference series sum converging to $\mcal N(0,\pi(|\sigma^\top\nabla \varphi|^2))$ in distribution, and a negligible remainder converging to $0$ in probability. In this proof, the $p$th ($p>2$) moment boundedness of the BEM method in the infinite time horizon and the regularity of the solution to the Poisson equation  play important roles, where the former has not been reported for SODEs with non-Lipschitz coefficients to the best our knowledge. 

To sum up, the contributions of this work are twofold. Firstly, we give the CLT for the temporal average of the BEM method, which generalizes the existing results to SODEs with super-linearly growing drift coefficients. Secondly, we prove the   $p$th ($p>2$) moment boundedness of the BEM method in the infinite time horizon for the original equation. The rest of this paper is organized as follows. In Section \ref{Sec2}, we give our assumptions and recall some basic properties of the exact solution. Section \ref{Sec3} presents our main results and proves the CLT for $\Pi_{\tau,\alpha}(h)$ with $\alpha\in(1,2)$, and Section \ref{Sec4} gives the proof of the CLT for  $\Pi_{\tau,2}(h)$. Some numerical tests are displayed to illustrate the theoretical results in Section \ref{Sec5}. Finally, we give the conclusions and future aspects in Section \ref{Sec6}.

\section{Preliminaries}\label{Sec2}
In this section, we give our main assumptions on the coefficients of \eqref{SDE} and present some basic properties for \eqref{SDE}. We begin with some notations.  Denote by $|\cdot|$  the $2$-norm of a vector or matrix, and by $\LL\cdot,\cdot\RR$ the scalar product of two vectors. Let $d,m,k\in\mbb N^+$ with $\mbb N^+$ denoting the set of positive integers. For matrix $A,B\in\mbb R^{d\times m}$, denote $\LL A,B\RR_{HS}:=\sum\limits_{i=1}^d\sum\limits_{j=1}^mA_{ij}B_{ij}$ and $\|A\|_{HS}:=\sqrt{\LL A,A\RR_{HS}}$. Let $\mscr B(\mbb R^d)$ stand for the set of all Borel sets of $\mbb R^d$. Denote by  $\mcal P(\mbb R^d)$ the space of all probability  measures on $\mbb R^d$. Denote $\mu(f)=\int_{\mbb R^d}f(x)\mu(\ud x)$ for $\mu\in\mcal P(\mbb R^d)$ and $\mu$-measurable function $f$.
For convenience, we set $\mcal F_t=\sigma(W(s),0\le s\le t)$. Moreover, $\overset{d}{\longrightarrow}$ denotes the convergence in distribution of random variables and $\overset{w}{\longrightarrow}$ denotes the weak convergence of probabilities in   $\mcal P(\mbb R^d)$.

Denote by $\mbf C(\mbb R^d;\mbb R^m)$ (resp. $\mbf C^{k}(\mbb R^d;\mbb R^m)$) the space consisting of continuous (resp. $k$th continuously differentiable) functions from $\mbb R^d$ to $\mbb R^m$. Let $\mbf C_b^k(\mbb R^d;\mbb R^m)$  stand for the set of bounded and $k$th continuously differentiable functions from $\mbb R^d$ to $\mbb R^m$ with bounded derivatives up to order $k$. Denote by $\mbf C_b(\mbb R^d;\mbb R^m)$ the set of bounded  and continuous functions from $\mbb R^d$ to $\mbb R^m$.
When no confusion occurs,  $\mbf C(\mbb R^d;\mbb R^m)$ is simply written as $\mbf C(\mbb R^d)$, and so are $\mbf C_b(\mbb R^d;\mbb R^m)$, $\mbf C^k(\mbb R^d;\mbb R^m)$ and  $\mbf C^k_b(\mbb R^d;\mbb R^m)$.
For $l\in\mbb N^+$, denote by $Poly(l,\mbb R^d)$ the set of functions growing polynomially with order $l$, i.e.,
$Poly(l,\mbb R^d):=\big\{g\in\mbf C(\mbb R^d;\mbb R):|g(x)-g(y)|\le  K(g)(1+|x|^{l-1}+|y|^{l-1})|x-y|~\text{for some}~K(g)>0 \big\}$.  For $f\in\mathbf C^k(\mathbb R^d;\mathbb R)$, denote by  $\nabla^k f(x)(\xi_1,\ldots,\xi_k)$ the $k$th order G\v ateaux derivative along the directions $\xi_1,\ldots,\xi_k\,\in\mathbb R^d$, i.e., $\nabla^k f(x)(\xi_1,\ldots,\xi_k)=\sum_{i_1,\ldots,i_k=1}^d\,\frac{\partial^k f(x)}{\partial x^{i_1}\cdots\partial x^{i_k}}\,\xi_1^{i_1}\cdots\xi_k^{i_k}$. For $f=(f_1,\ldots,f_m)^\top\in\mathbf C^k(\mathbb R^d; \mathbb R^{m})$, denote $\nabla ^kf(x)(\xi_1,\ldots,\xi_k)=(\nabla ^kf_1(x)(\xi_1,\ldots,\xi_k),\ldots,\nabla ^kf_m(x)(\xi_1,\ldots,\xi_k))^\top$. The  G\v ateaux derivative for a matrix-valued function is defined as previously. For $f\in \mbf C^k(\mbb R^d;\mbb R)$ the notation $\nabla ^kf(x)$ is viewed as a tensor, i.e., a multilinear form defined on $(\mbb R^d)^{\otimes^k}$. And  $\|\cdot\|_{\otimes}$ denotes the norm of a tensor. Throughout this paper, let $K(a_1,a_2,...,a_m)$ denote some generic constant dependent on the parameters $a_1,a_2,...,a_m$ but independent of the step-size $\tau$, which may
vary for each appearance.

\subsection{Settings}\label{Sec2.1}
Let us first  give the assumptions on $b$ and $\sigma$.

\begin{assum}\label{assum1}
	There exist constants $L_1, L_2\in(0,+\infty)$ such that
	\begin{gather*}
		\|\sigma(u_1)-\sigma(u_2)\|_{HS}\le L_1|u_1-u_2|\quad\forall~u_1,u_2\in\mbb R^d,\label{sigLip}\\
		\|\sigma(u)\|_{HS}\le L_2\quad\forall~u\in\mbb R^d.\label{sigbounded}
	\end{gather*}
\end{assum}

\begin{assum}\label{assum2}
	There exist $c_1>\frac{15}{2}L_1^2$, $L_3>0$ and $q\ge 1$ such that
	\begin{gather}
		\LL u_1-u_2,b(u_1)-b(u_2)\RR\le -c_1|u_1-u_2|^2 \quad\forall~u_1,u_2\in\mbb R^d,\nonumber\\
		|b(u_1)-b(u_2)|\le L_3(1+|u_1|^{q-1}+|u_2|^{q-1})|u_1-u_2|\quad\forall~u_1,u_2\in\mbb R^d.\label{bsuper}
	\end{gather}
\end{assum}
The above two assumptions ensure the well-posedness of \eqref{SDE}; see e.g., \cite{Liuwei23}. And the generator of \eqref{SDE} is given by
\begin{align}\label{generator}
	\mcal Lf(x)=\LL \nabla  f(x),b(x)\RR+\frac{1}{2}\LL \nabla^2f(x),\sigma(x)\sigma(x)^\top\RR_{HS}, ~f\in\mbf C^2(\mbb R^d;\mbb R).
\end{align}
Notice that $\text{trace}(\nabla^2 f(x)\sigma(x)\sigma(x)^\top)=\LL \nabla^2 f(x),\sigma(x)\sigma(x)^\top\RR_{HS}$.

As an immediate result of \eqref{bsuper}, 
	\begin{align}\label{bgrowth}
	|b(u)|\le L_4(1+|u|^q)\quad\forall~u\in\mbb R^d
	\end{align}
	for some $L_4>0$. In addition, it is straightforward to conclude from Assumptions \ref{assum1}-\ref{assum2} that  for any $l_2>0$,
	\begin{gather}
		2\LL u_1-u_2,b(u_1)-b(u_2)\RR+15\|\sigma(u_1)-\sigma(u_2)\|_{HS}^2\le -L_5|u_1-u_2|^2\quad\forall~u_1,u_2\in\mbb R^d,\label{couple1}\\
		2\LL u,b(u)\RR+l_2\|\sigma(u)\|_{HS}^2\le -c_1|u|^2+\frac{1}{c_1}|b(0)|^2+l_2L_2^2\quad\forall~u\in\mbb R^d, \label{ubu}
	\end{gather}
where $L_5:=2c_1-15L_1^2$. Note that Assumptions \ref{assum1}-\ref{assum2} in this paper imply Assumptions $2.1$-$2.4$ in \cite{Liuwei23}, by taking $A=-\varepsilon I_d$, $f(x)=b(x)+\varepsilon x$ and $g(x)=\sigma(x)$ in \cite[Eq.\  (2)]{Liuwei23} with $\varepsilon$ small enough. Thus, all conclusions in \cite{Liuwei23} apply to our case provided that Assumptions \ref{assum1}-\ref{assum2} hold.

In order to give the regularity of the solution to the Poisson equation, we need the following assumption.
\begin{assum}\label{assum3}
Let $\sigma\in\mbf C_b^4(\mbb R^d)$ and $b\in\mbf C^4(\mbb R^d)$. In addition,  there exist $q'\ge 1$ and $L_6>0$ such that for $i=1,2,3,4$,
\begin{align*}
	\|\nabla^i b(u)\|_{\otimes}\le L_6(1+|u|^{q'})\quad\forall~ u\in\mbb R^d.
\end{align*} 
\end{assum}

\begin{rem}
	Under Assumptions \ref{assum1}-\ref{assum3}, it holds that
	\begin{align}\label{couple2}
		2\LL v,\nabla b(u)v\RR+15\|\nabla \sigma(u)v\|^2_{HS}\le -L_5|v|^2\quad\forall~ u,v\in\mbb R^d.
	\end{align}
In fact, it follows from \eqref{couple1} that for any $u,v\in\mbb R^d$ and $t\in\mbb R$, $2t\LL v,b(u+tv)-b(u)\RR+15\|\sigma(u+tv)-\sigma(u)\|^2_{HS}\le -L_5t^2|v|^2$. Then the Taylor expansion yields that for any $t\in\mbb R$,   $2t^2\LL v,\nabla b(u)v\RR+15t^2\|\nabla\sigma(u)v\|^2_{HS}+\mcal O(t^3)\le -L_5t^2|v|^2$, which implies \eqref{couple2}.
\end{rem}

Next, we recall some basic knowledge about the invariant measure and ergodicity. Denote by $X^{s,x}(t)$ the solution to \eqref{SDE} at time $t$, starting from $X(s)=x$. Especially, denote $X^x(t):=X^{0,x}(t)$. Let $\pi_t(x,\cdot)$ denote the transition probability kernel of $\{X(t)\}_{t\ge 0}$, i.e., $\pi_t(x,A)=\mbf P(X^x(t)\in A)$ for any $A\in \mscr B(\mbb R^d)$. For any $\phi\in \mbf B_b(\mbb R^d)$ and $t\ge0$, define the operator $P_t:\mbf B_b(\mbb R^d)\to\mbf B_b(\mbb R^d)$ by $(P_t\phi)(x):=\mbf E\phi(X^x(t))=\int_\mbb R\phi(y)\pi_t(x,\ud y)$. Then, $\{P_t\}_{t\ge 0}$ is a  Markov semigroup on $\mbf B_b(\mbb R^d)$. Here, $\mbf B_b(\mbb R^d)$ is the space of all bounded and measureale functions.
A probability measure $\mu\in\mcal P(\mbb R^d)$ is called an invariant measure of $\{X(t)\}_{t\ge 0}$ or $\{P_t\}_{t\ge0}$, if
\begin{align}
	\int_{\mbb R^d} P_t\phi(x)\mu (\ud x)=	\int_{\mbb R^d} \phi(x)\mu (\ud x)\quad\forall~\phi\in\mbf B_b(\mbb R^d),~t\ge 0.\label{invariant}
\end{align}
Further, an invariant measure $\mu$ is called an ergodic measure of $\{X(t)\}_{t\ge 0}$ or $\{P_t\}_{t\ge0}$, if for any $\phi\in\mbf B_b(\mbb R^d)$,
\begin{align*}
\lim_{T\to+\infty}\frac{1}{T}\int_0^TP_t\phi(x)\ud t=\int_{\mbb R^d} \phi(x)\mu (\ud x)\quad \text{in}~\mbf L^2(\mbb R^d,\mu),
\end{align*}
where $\mbf L^2(\mbb R^d,\mu)$ is the space of all square integrable functions with respect to (w.r.t.) $\mu$. Especially, if $\mu$ is the unique invariant measure of $\{X(t)\}_{t\ge 0}$, then $\mu$ is also the ergodic measure. 
We refer readers to \cite{Wangxu} for more details.

\begin{pro}\label{Xtesti}
	Let Assumptions \ref{assum1}-\ref{assum2} hold. Then, the following hold.
	\begin{itemize}
	\item[(1)] For any $p\ge 1$, $\sup\limits_{t\ge 0}\mbf E|X^x(t)|^p\le K(p)(1+|x|^p)$.
	\item[(2)] For any $t,s\ge 0$, $\big(\mbf E|X^x(t)-X^x(s)|^2\big)^{\frac12}\le K(1+|x|^q)|t-s|^{\frac12}$.
\item[(3)] For any $t\ge 0$, $\big(\mbf E|X^x(t)-X^y(t)|^2\big)^{\frac12}\le |x-y|e^{-\frac{L_5t}{2}}$.
	\end{itemize}
\end{pro}
The first and second conclusions come from \cite[Proposition 3.1-3.2]{Liuwei23}. And the third conclusion can be obtained by applying the It\^o formula. In addition, \cite[Theorem 3.1]{Liuwei23} gives the ergodicity for $\{X(t)\}_{t\ge 0}$.

\begin{pro}\label{Xtergodic}
	Let Assumptions \ref{assum1}-\ref{assum2} hold. Then we have the following.\\
	(1) $\{X(t)\}_{t\ge 0}$ admits a unique invariant measure $\pi\in\mcal P(\mbb R^d)$.\\
	(2) For any $p\ge 1$, $\pi(|\cdot|^p)<+\infty$.\\
	(3) There is $\lambda_1>0$ such that for any $f\in Poly(l,\mbb R^d)$, $l\ge1$ and $t\ge 0$,
	\begin{gather}
	\big|\mbf Ef(X^x(t))-\pi(f)\big|\le K(f)(1+|x|^l)e^{-\lambda_1t}.\label{pif}
	\end{gather}
\end{pro}
\begin{proof}
It follows from \cite[Theorem 3.1]{Liuwei23} that $\{X(t)\}_{t\ge0}$ admits a unique invariant measure $\pi\in\mcal P(\mbb R^d)$, and $\pi_t(x,\cdot)\overset{w}{\longrightarrow}\pi$ as $t\to+\infty$ for any $x\in\mbb R^d$. Especially, $\pi_t(0,\cdot)\overset{w}{\longrightarrow}\pi$, which implies that for any $M>0$,
\begin{align*}
	\int_{\mbb R^d}(|x|^p\wedge M)\pi(\ud x)&=\lim_{t\to+\infty}\int_{\mbb R^d}(|x|^p\wedge M)\pi_t(0,\ud x)\\
	&\le M\wedge  \limsup_{t\to+\infty}\mbf E|X^0(t)|^p\le K,
\end{align*}
where we used $|\cdot|^p\wedge M\in\mbf C_b(\mbb R^d)$ and Proposition \ref{Xtesti}(1). Then the Fatou lemma gives
\begin{align*}
\pi(|\cdot|^p)=\int_{\mbb R^d}|x|^p\;\pi(\ud x)\le \liminf_{M\to+\infty}\int_{\mbb R^d}(|x|^p\wedge M)\pi (\ud x)\le K.	
\end{align*}

For any $M>0$ and $f\in Poly(l,\mbb R^d)$, it holds that $f\wedge M\in\mbf C_b(\mbb R^d)$.  Accordingly, it follows from the definition of the invariant measure (see \eqref{invariant}) that
\begin{align*}
	\pi(f\wedge M)=\int_{\mbb R^d} P_t(f\wedge M)(y)\pi (\ud y).
\end{align*} 
Thus,  using Proposition \ref{Xtesti}(2), the H\"older inequality, the fact $|a\wedge b-a\wedge c|\le |b-c|$ and the second conclusion, we conclude that for any $M>0$,
\begin{align*}
	&\;|\mbf E(f(X^x(t))\wedge M)-\pi(f\wedge M)|=\Big|P_t(f\wedge M)(x)-\int_{\mbb R^d} P_t(f\wedge M)(y)\pi (\ud y)\Big|\\
	= &\; \Big|\int_{\mbb R^d}\big[P_t(f\wedge M)(x)-P_t(f\wedge M)(y)\big]\pi (\ud y)\Big|\\
	\le &\;\int_{\mbb R^d}\big|\mbf E(f(X^x(t))\wedge M)-\mbf E(f(X^y(t))\wedge M)\big|\pi (\ud y)\\
	\le &\;\int_{\mbb R^d}\mbf E|f(X^x(t))-f(X^x(y))|\pi(\ud y)\\
	\le&\; K(f)\int_{\mbb R^d}(1+(\mbf E|X^x(t)|^{2l-2})^{\frac12}+(\mbf E|X^y(t)|^{2l-2})^{\frac12})\big(\mbf E|X^x(t)-X^y(t)|^2\big)^{\frac12}\pi (\ud y)\\
	\le&\;K(f)e^{-\frac{L_5}{2}t}\int_{\mbb R^d}(1+|x|^{l-1}+|y|^{l-1})|x-y|\pi(\ud y)\\
	\le &\; K(f)e^{-\frac{L_5}{2}t}(1+|x|^l).
\end{align*}
The above formula and the monotone convergence theorem lead to \eqref{pif}, which completes the proof.
\end{proof}

\section{Main results}\label{Sec3}
In this section, we give our main result, i.e., the CLT for  the temporal average $\Pi_{\tau,\alpha}(h)$ of the BEM method used to approximate the ergodic limit $\pi(h)$.  The BEM method has been widely applied to approximating  SODEs or SPDEs with non-Lipschitz coefficients; see e.g., \cite{CaiGanHu23,HHKW20,LiuQiao20}  and references therein.
Let $\tau>0$ be the temporal step-size. The BEM method for \eqref{SDE} reads
\begin{align}\label{BEM}
		\bar{X}_{n+1}=\bar{X}_n+b(\bar{X}_{n+1})\tau+\sigma(\bar{X}_n)\Delta W_n,\quad n=0,1,2,\ldots,
\end{align} 
where $\Delta W_n:=W(t_{n+1})-W(t_n)$ with $t_n=n\tau$. We denote by $\bar{X}^{k,x}_n$ the solution to \eqref{BEM} at the $n$th step provided $\bar{X}_k=x$. Especially, denote $\bar{X}_n^x:=\bar{X}^{0,x}_n$, i.e., the solution to \eqref{BEM} with the initial value $x\in\mbb R^d$. 

The following are some known results about \eqref{BEM}, which can  be found in Lemmas $4.1$-$4.2$ and Theorems $4.2$ in \cite{Liuwei23}.
\begin{pro}\label{BEMesti}
	Let Assumptions \ref{assum1}-\ref{assum2} hold and $\tau$ sufficiently small.  Then the following properties hold.
	\begin{itemize}
		\item [(1)] $\sup\limits_{n\ge 0}\mbf E|\bar{X}^x_n|^2\le K(1+|x|^2)$.
		\item[(2)] There is $\xi_1>0$ such that for any $n\ge 0$,
		$\big(\mbf E|\bar{X}^x_n-\bar{X}^y_n|^2\big)^{\frac12}\le K|x-y|e^{-\xi_1n\tau}$.
		\item [(3)] $\sup\limits_{n\ge 0}\mbf E|X^x(t_n)-\bar{X}^x_n|^2\le K(x)\tau.$
	\end{itemize}
\end{pro}

Recall that the temporal average of the BEM method is
\begin{align*}
\Pi_{\tau,\alpha}(h)=\frac{1}{\tau^{-\alpha}}\sum_{k=0}^{\tau^{-\alpha}-1}h(\bar{X}^x_k), \quad\alpha\in(1,2],~h\in \mbf C^4_b(\mbb R^d).
\end{align*}
Define the function $\varphi:\mbb R^d\to\mbb R$ by
\begin{align}
	\varphi(x)=-\int_0^\infty \mbf E\big(h(X^x(t))-\pi(h)\big)\ud t,~x\in\mbb R^d \label{eq10},
\end{align}
which is indeed a solution to the Poisson equation $\mcal L\varphi=h-\pi(h)$ due to Lemma \ref{varphi}. Then we have the following CLT for $\Pi_{\tau,\alpha}(h)$, $\alpha\in(1,2)$.

\begin{theo}\label{CLT1}
	Let Assumptions \ref{assum1}-\ref{assum3} hold and $h\in\mbf C_b^4(\mbb R^d)$. \\
(1)	Let  $\{Y_n\}_{n\ge0}$ be a numerical solution for \eqref{SDE}. Suppose that there is $K>0$ independent of $\tau$ such that
	\begin{align}
		\sup_{n\ge 0}\mbf E|X(t_n)-Y_n|^2\le K\tau.\label{strong}
	\end{align}
Then for any $\alpha\in(1,2)$,
\begin{align}
	\frac{1}{\tau^{\frac{\alpha-1}{2}}}\Big(\frac{1}{\tau^{-\alpha}}\sum_{k=0}^{\tau^{-\alpha}-1}h(Y_k)-\pi(h)\Big)\overset{d}{\longrightarrow}\mcal N(0,\pi(|\sigma^\top\nabla \varphi|^2)) \quad\text{as}~\tau\to 0.\label{Eq9}
\end{align}
(2) For any $\alpha\in(1,2)$ and $x\in\mbb R^d$,
	\begin{align}
		\frac{1}{\tau^{\frac{\alpha-1}{2}}}(\Pi_{\tau,\alpha}(h)-\pi(h))\overset{d}{\longrightarrow}\mcal N(0,\pi(|\sigma^\top\nabla \varphi|^2)) \quad\text{as}~\tau\to 0.\label{Eq10}
	\end{align}
\end{theo}
\begin{proof}
	Let $\varphi$ be that in \eqref{eq10}. By  Lemma \ref{varphi}, it holds that $\varphi\in\mbf C^3(\mbb R^d)$ and
	\begin{align*}
		\mcal L\varphi=h-\pi(h). 
	\end{align*}
It follows from \cite[Thoerem 2.1]{FCLT82} that the CLT holds for \eqref{SDE}, i.e.,
\begin{align*}
	\frac{1}{\sqrt{T}}\int_0^T\big(h(X(t))-\pi(h)\big)\ud t\overset{d}{\longrightarrow}\mcal N\big(0,-2\pi(\varphi\mcal L\varphi)\big)\quad\text{as}~T\to \infty.
\end{align*}
By \eqref{generator} and a direct computation, 
\begin{align*}
	\varphi\mcal L\varphi=\frac{1}{2}\mcal L\varphi^2-\frac{1}{2}|\sigma^\top\nabla\varphi|^2.
\end{align*} 
Since $\varphi^2$ belongs to the domain of $\mcal L$, $\pi(\mcal L\varphi^2)=0$ due to \cite[Eq. (2.6)]{FCLT82}.  Combining the above relations, we have
\begin{align}
	\frac{1}{\sqrt{T}}\int_0^T\big(h(X(t))-\pi(h)\big)\ud t\overset{d}{\longrightarrow}\mcal N\big(0,\pi(|\sigma^\top\nabla \varphi|^2)\big)\quad\text{as}~T\to \infty.\label{eq18}
\end{align}

Notice that
\begin{align*}
	&\;\frac{1}{\tau^{\frac{\alpha-1}{2}}}\Big(\frac{1}{\tau^{-\alpha}}\sum_{k=0}^{\tau^{-\alpha}-1}h(Y_k)-\pi(h)\Big)\\
	=&\;\frac{1}{\tau^{\frac{\alpha-1}{2}}}\Big(\frac{1}{\tau^{-\alpha}}\sum_{k=0}^{\tau^{-\alpha}-1}h(Y_k)-\tau^{\alpha-1}\int_0^{\tau^{1-\alpha}}h(X(t))\ud t\Big)+\tau^{\frac{\alpha-1}{2}}\int_0^{\tau^{1-\alpha}}\big(h(X(t))-\pi(h)\big)\ud t\\
	=:&\;J_1(\tau)+J_2(\tau).
\end{align*}
By \eqref{eq18} and $\alpha>1$, $J_2(\tau)\overset{d}{\longrightarrow}\mcal N\big(0,\pi(|\sigma^\top\nabla \varphi|^2)\big)$ as $\tau\to 0$. Denoting $N=\tau^{-\alpha}$,
 we use Proposition \ref{Xtesti}(2), \eqref{strong} and $h\in \mbf C_b^1(\mbb R^d)$ to get
\begin{align*}
\mbf E|J_1(\tau)|=&\;\frac{1}{\tau^{\frac{\alpha-1}{2}}}\mbf E\Big|\frac{1}{N}\sum_{k=0}^{N-1}h(Y_k)-\frac{1}{N\tau}\sum_{k=0}^{N-1}\int_{k\tau}^{(k+1)\tau}h(X(t))\ud t\Big|\\
\le &\;\frac{1}{\tau^{\frac{\alpha-1}{2}}}\frac{1}{N}\sum_{k=0}^{N-1}\mbf E|h(Y_k)-h(X(t_k))|+\frac{1}{\tau^{\frac{\alpha-1}{2}}}\frac{1}{N\tau}\sum_{k=0}^{N-1}\int_{k\tau}^{(k+1)\tau}\mbf E|h(X(t))-h(X(t_k))|\ud t\\
\le&\;K(h)\frac{1}{\tau^{\frac{\alpha-1}{2}}}\sup_{k\ge0}(\mbf E|Y_k-X(t_k)|^2)^{\frac12}+K(h)\frac{1}{\tau^{\frac{\alpha-1}{2}}}\frac{1}{N\tau}\sum_{k=0}^{N-1}\int_{k\tau}^{(k+1)\tau}(\mbf E|X(t)-X(t_k)|^2)^{\frac12}\ud t\\
\le &\;K(h)\frac{1}{\tau^{\frac{\alpha-1}{2}}}\tau^{\frac{1}{2}}=K(h)\tau^{\frac{2-\alpha}{2}}.
\end{align*}
Thus, $\lim\limits_{\tau\to0}\mbf E|J_1(\tau)|=0$ due to $\alpha<2$, which implies that $J_1(\tau)$ converges to $0$ in probability. Thus, \eqref{Eq9} follows by applying the Slutsky theorem.

Finally, \eqref{Eq10} holds as a special case of \eqref{Eq9} due to Proposition \ref{BEMesti}(3). Thus, the proof is complete. 
\end{proof}

\begin{rem}\label{rem1}\phantom{text}\\
(1)	It is observed that $$\frac{1}{\tau^{\frac{\alpha-1}{2}}}(\Pi_{\tau,\alpha}(h)-\pi(h))=\frac{1}{\tau^{\frac{1-\alpha}{2}}}\sum_{k=0}^{\tau^{-\alpha}-1}\big(h(\bar{X}^x_k)-\pi(h)\big)\tau,$$
which can be viewed as a numerical approximation of $\frac{1}{\sqrt{T}}\int_0^T(h(X^x(t))-\pi(h))\ud t$ with $T(\tau)=N\tau$ and $N=\tau^{-\alpha}$. Thus, $\alpha>1$ is required such that $\lim\limits_{\tau\to0}T(\tau)=+\infty$,  which coincides with the CLT for $\{X(t)\}_{t\ge0}$.\\
(2) In fact, we give the CLT of the temporal average for a class of numerical methods satisfying \eqref{strong} for $\alpha\in(1,2)$. We guess that there may be some non-ergodic numerical method whose temporal average satisfies the CTL in view of Theorem \ref{CLT1}(1).
\end{rem}

We close the section by presenting the CLT for $\Pi_{\tau,2}(h)$.
\begin{theo}\label{CLT2}
	Let Assumptions \ref{assum1}-\ref{assum3} hold and $h\in \mbf C^4_b(\mbb R^d)$. Then for any $x\in\mbb R^d$,
	\begin{align*}
		\frac{1}{\sqrt{\tau}}(\Pi_{\tau,2}(h)-\pi(h))\overset{d}{\longrightarrow}\mcal N(0,\pi(|\sigma^\top\nabla \varphi|^2)) \quad\text{as}~\tau\to 0.
	\end{align*}
\end{theo}
As is pointed out in the introduction, the proof idea of Theorem \ref{CLT1} does apply to the case $\alpha=2$. Instead, we will use the Poisson equation $\mcal L \varphi=h-\pi(h)$ to give a good decomposition of $\Pi_{\tau,2}(h)$, on basis of which the CLT of $\Pi_{\tau,2}(h)$ can be established.
We postpone the proof of Theorem \ref{CLT2} to the next section.

\section{Proof of Theorem \ref{CLT2}}\label{Sec4}
\subsection{Auxiliary results}
Notice that \cite{Liuwei23} gives the second moment boundedness of the BEM method, i.e, Proposition \ref{BEM}(1). However, in order to give the CLT for $\Pi_{\tau,2}(h)$, the $p$th ($p>2$) moment boundedness in the infinite time  horizon is indispensable.  We also refer interested readers to \cite{trunEM19} for the $p$th ($p>2$) moment boundedness in the infinite time horizon  for the truncated Euler Maruyama method.
\begin{theo}\label{pth-moment}
	Suppose that Assumptions \ref{assum1}-\ref{assum2} hold. Then for any $r\ge 1$ and $\tau\le 1$,
	\begin{align}\label{moment}
	\sup_{n\ge0}\mbf E|\bar{X}^x_n|^{r}\le K(r)(1+|x|^{r}).
	\end{align}
\end{theo}

\begin{proof}
	It is sufficient to show that for any positive integer $p$,
	\begin{align}\label{2pmoment}
		\sup_{n\ge0}\mbf E|\bar{X}^x_n|^{2p}\le K(p)(1+|x|^{2p}),
	\end{align}
 in view of the H\"older inequality, which will be derived  via  mathematical induction. 
	
	By \eqref{BEM} and \eqref{ubu},
	\begin{align}
	&\;|\bar{X}^x_{n+1}|^2-|\bar{X}^x_n|^2+|\bar{X}^x_{n+1}-\bar{X}^x_{n}|^2=2\LL \bar{X}^x_{n+1},\bar{X}^x_{n+1}-\bar{X}^x_{n}\RR\nonumber\\
	=&\;2\LL\bar{X}^x_{n+1},b(\bar{X}^x_{n+1})\RR\tau+2\LL\bar{X}^x_{n+1}-\bar{X}^x_{n},\sigma(\bar{X}^x_{n})\Delta W_n\RR+2\LL\bar{X}^x_{n},\sigma(\bar{X}^x_{n})\Delta W_n\RR\nonumber\\
	\le &\; -c_1\tau|\bar{X}^x_{n+1}|^2+K\tau+|\bar{X}^x_{n+1}-\bar{X}^x_{n}|^2+\|\sigma(\bar{X}^x_{n})\|^2_{HS}|\Delta W_n|^2+2\LL\bar{X}^x_{n},\sigma(\bar{X}^x_{n})\Delta W_n\RR,\label{Eq11}
	\end{align}
which together with the boundedness of $\sigma$ yields
\begin{align}\label{eq1}
	(1+c_1\tau)|\bar{X}^x_{n+1}|^2\le |\bar{X}^x_n|^2+K\tau+L_2^2|\Delta W_n|^2+2\LL\bar{X}^x_{n},\sigma(\bar{X}^x_{n})\Delta W_n\RR.
\end{align}	
Noting that $\mbf E\LL\bar{X}^x_{n},\sigma(\bar{X}^x_{n})\Delta W_n\RR=0$, we have
\begin{align*}
	\mbf E|\bar{X}^x_{n+1}|^2\le \frac{1}{1+c_1\tau}\mbf E|\bar{X}^x_n|^2+\frac{K\tau}{1+c_1\tau}.
\end{align*}	
By iteration, we arrive at 
\begin{align*}
	\mbf E|\bar{X}^x_n|^2\le \frac{1}{(1+c_1\tau)^n}|x|^2+K\tau\sum_{i=1}^\infty\frac{1}{(1+c_1\tau)^i}\le |x|^2+K.
\end{align*}
Thus, \eqref{2pmoment} holds for $p=1$. Now, we assume that 
\begin{align}\label{induction}
	\sup_{n\ge0}\mbf E|\bar{X}^x_n|^{2(p-1)}\le K(p)(1+|x|^{2(p-1)}),~p\ge 2.
\end{align}	
\noindent It remains to prove 	$\sup\limits_{n\ge0}\mbf E|\bar{X}^x_n|^{2p}\le K(p)(1+|x|^{2p})$.

In fact, using \eqref{eq1} and the inequality $(1+x)^\alpha\ge 1+\alpha x$, $\alpha\ge 1$, $x>-1$ leads to
\begin{align}\label{eq2}
	(1+pc_1\tau)|\bar{X}^x_{n+1}|^{2p}\le \big(|\bar{X}^x_{n}|^2+2\LL\bar{X}^x_{n},\sigma(\bar{X}^x_{n})\Delta W_n \RR+K(\tau+|\Delta W_n|^2)\big)^p.
\end{align}
Notice that
\begin{align*}
	&\;\big(|\bar{X}^x_{n}|^2+2\LL\bar{X}^x_{n},\sigma(\bar{X}^x_{n})\Delta W_n \RR+K(\tau+|\Delta W_n|^2)\big)^p\\
	=&\;\sum_{i_1=0}^{p}\sum_{i_2=0}^{p-i_1}C_p^{i_1}C_{p-i_1}^{i_2}2^{i_2}K^{p-(i_1+i_2)}|\bar{X}^x_n|^{2i_1}\LL\bar{X}^x_{n},\sigma(\bar{X}^x_{n})\Delta W_n \RR^{i_2}(\tau+|\Delta W_n|^2)^{p-(i_1+i_2)}\\
	=&\; |\bar{X}^x_n|^{2p}+\sum_{i_1=0}^{p-1}\sum_{i_2=0}^{p-i_1-1}C_p^{i_1}C_{p-i_1}^{i_2}2^{i_2}K^{p-(i_1+i_2)}S_{n,i_1,i_2}+\sum_{i=0}^{p-1}C_p^{i}2^{p-i}T_{n,i},
\end{align*}
where
\begin{gather*}
S_{n,i_1,i_2}:=|\bar{X}^x_n|^{2i_1}\LL\bar{X}^x_{n},\sigma(\bar{X}^x_{n})\Delta W_n \RR^{i_2}(\tau+|\Delta W_n|^2)^{p-(i_1+i_2)},~i_1\in[0,p-1],~i_2\in[0,p-i_1-1],\\
T_{n,i}:=|\bar{X}^x_n|^{2i}\LL\bar{X}^x_{n},\sigma(\bar{X}^x_{n})\Delta W_n \RR^{p-i},~i\in[0,p-1].
\end{gather*} 
For any $i_1\in[0,p-1],~i_2\in[0,p-i_1-1]$, it follows from  the independence of $\Delta W_n$ and $\bar{X}^x_n$,  the boundedness of $\sigma$, the H\"older inequality and \eqref{induction}  that for $\tau\le 1$,
\begin{align*}
	|\mbf ES_{n,i_1,i_2}|&\le K(p) \mbf E|\bar{X}^x_n|^{2i_1+i_2}\mbf E\big[|\Delta W_n|^{i_2}(\tau+|\Delta W_n|^2)^{p-(i_1+i_2)}\big]\\
	&\le K(p)\big(\mbf E|\bar{X}^x_n|^{2p-2}\big)^{\frac{2i_1+i_2}{2p-2}}\tau\le K(p)(1+|x|^{2p-2})\tau.
\end{align*}
Next we estimate $|\mbf ET_{n,i}|$ for $i=0,\ldots,p-1$. 
Notice that the property of conditional expectations (see, e.g., \cite[Lemma 2.6]{LiuMao15}) leads to
\begin{align*}
	\mbf ET_{n,p-1}&=\mbf E\big[\mbf E_n\big(|\bar{X}^x_n|^{2p-2}\LL\bar{X}^x_{n},\sigma(\bar{X}^x_{n})\Delta W_n \RR\big)\big]\\
	&=\mbf E\big[\big(\mbf E(|y|^{2p-2}\LL y,\sigma(y)\Delta W_n \RR)\big)\big|_{y=\bar{X}^x_n}\big]=0.
\end{align*}
For $i=0,\ldots,p-2$, applying \eqref{induction}, the boundedness of $\sigma$ and  the H\"older inequality, we get
\begin{align*}
	|\mbf ET_{n,i}|\le K(p)\mbf E|\bar{X}^x_n|^{p+i}\mbf E|\Delta W_n|^{p-i}\le K(p) \big(\mbf E|\bar{X}^x_n|^{2p-2})^{\frac{p+i}{2p-2}}\tau^{\frac{p-i}{2}}\le K(p)(1+|x|^{2p-2})\tau.
\end{align*}
Combining the above formulas gives
\begin{align*}
	\mbf E\big(|\bar{X}^x_{n}|^2+2\LL\bar{X}^x_{n},\sigma(\bar{X}^x_{n})\Delta W_n \RR+K(\tau+|\Delta W_n|^2)\big)^p\le\mbf E|\bar{X}^x_n|^{2p}+K(p)(1+|x|^{2p-2})\tau,
\end{align*}
which along with \eqref{eq2} yields
\begin{align}
	\mbf E|\bar{X}^x_{n+1}|^{2p}\le \frac{1}{1+pc_1\tau}\mbf E|\bar{X}^x_n|^{2p}+K(p)\frac{(1+|x|^{2p-2})\tau}{1+pc_1\tau}.\label{Eq12}
\end{align}
Then by iteration, we deduce 
\begin{align*}
	\mbf E|\bar{X}^x_{n}|^{2p}\le \frac{1}{(1+pc_1\tau)^n}|x|^{2p}+K(p)(1+|x|^{2p-2})\tau\sum_{i=1}^\infty \frac{1}{(1+pc_1\tau)^i}\le K(p)(1+|x|^{2p}).	
\end{align*}
Thus, \eqref{2pmoment} holds by mathematical induction and the proof is complete.
\end{proof}

\begin{pro}\label{BEMergodic}
	Let Assumptions \ref{assum1}-\ref{assum2} hold and $\tau$ be sufficiently small. Then  the BEM method \eqref{BEM} admits	a unique invariant measure $\pi_\tau\in\mcal P(\mbb R^d)$. Moreover,
	for any $f\in Poly(l,\mbb R^d)$, $l\ge1$ and $n\ge 0$,
	\begin{align}
		\big|\mbf Ef(\bar{X}^x_n)-\pi_\tau(f)\big|&\le K(f)(1+|x|^l)e^{-\xi_1n\tau},\quad x\in\mbb R^d,~n\ge 0, \label{pitau}\\
		|\pi_\tau(f)-\pi(f)|&\le K(f)\tau^{\frac12}.\label{pitaupi}
	\end{align}
\end{pro}
\begin{proof}
	As is shown in \cite[Theorem 4.1]{Liuwei23}, $\{\bar{X}_n\}_{n\ge0}$ admits a unique invariant measure $\pi_\tau$, and     $\bar{X}_{n}^x\overset{d}{\longrightarrow}\pi_\tau$ for any $x\in\mbb R^d$. Similar to the proof of \eqref{pif}, one can derive \eqref{pitau} based on Proposition \ref{BEMesti}(2) and Theorem \ref{pth-moment}. As for \eqref{pitaupi}, it follows from $f\in Poly(l,\mbb R^d)$,  \eqref{pitau}, Theorem \ref{pth-moment}, Proposition \ref{Xtesti}(1), Proposition \ref{BEMesti}(3) and \eqref{pif}  that for any $n\ge 0$ and $\tau\ll 1$,
	\begin{align*}
		&|\pi_\tau(f)-\pi(f)|\le |\pi_\tau(f)-\mbf Ef(\bar{X}^0_n)|+|\mbf Ef(\bar{X}^0_n)-\mbf Ef(X^0(t_n))|+|\mbf Ef(X^0(t_n))-\pi(f)|\\
		\le& K(f)e^{-\xi_1n\tau}+K(f)(1+(\mbf E|\bar X^0_n|^{2l-2})^{\frac12}+(\mbf E|X^0(t_n)|^{2l-2})^{\frac12}) (\mbf E|\bar X^0_n-X^0(t_n)|^2)^{\frac12}+K(f)e^{-\lambda_1 t_n}\\
		\le &K(f)(e^{-\xi_1n\tau}+e^{-\lambda_1 t_n})+K(f)\tau^{\frac12}.
	\end{align*}
Letting $n\to\infty$ in the above formula yields \eqref{pitaupi}, which finishes the proof.
\end{proof}

In order to prove the  CLT for $\Pi_{\tau,2}(h)$, we need  to give the regularity of $\varphi$. This can be done through a probabilistic approach by means of  mean-square derivatives of $\{X^x(t)\}_{t\ge0}$ w.r.t. the initial value $x$.
For any $x,y_i\in\mbb R^d$, $i=1,2,3,4$, denote by $\eta^x_{y_1}(t)$ the mean-square derivative of $X^x(t)$ along with the direction $y_1$, i.e., $\eta^x_{y_1}(t)=\lim\limits_{\varepsilon\to0}\frac{1}{\varepsilon}(X^{x+\varepsilon y_1}(t)-X^x(t))$ in $\mbf L^2(\Omega;\mbb R^d)$. Further, denote $\eta^x_{y_1,y_2}(t):=\lim\limits_{\varepsilon\to0}\frac{1}{\varepsilon}(\eta^{x+\varepsilon y_2}_{y_1}(t)-\eta^x_{y_1}(t))$ in $\mbf L^2(\Omega;\mbb R^d)$, i.e., $\eta^x_{y_1,y_2}(t)$ is the second mean-square derivative of $X^x(t)$ along with the direction $y_1$ and $y_2$. $\eta^x_{y_1,y_2,y_3}(t)$ and $\eta^x_{y_1,y_2,y_3,y_4}(t)$ are defined similarly. 
%Since Assumptions \ref{assum1}-\ref{assum3} implies the Hypotheses $1.1$ and $1.2$ in \cite[Chapter 1]{meandiff01}, it follows  from \cite[Theorem 1.3.6]{meandiff01} that $X^x(t)$ is four times means-square differentiable. 
%Thus, $\eta^x_{y_1}(t)$, $\eta^x_{y_1,y_2}(t)$, $\eta^x_{y_1,y_2,y_3}(t)$ and $\eta^x_{y_1,y_2,y_3,y_4}(t)$ are well-defined.
We refer readers to  \cite{meandiff01,WangZhao} for more details about the  mean-square  differentiablity of SDEs w.r.t. initial values.

\begin{lem}\label{uestimate}
Suppose that Assumptions \ref{assum1}-\ref{assum3} hold. Then there exist $C_1,C_2>0$ and $\kappa_i>0$, $i=1,2,3$ such that for any $x,y_i\in\mbb R^d$, $i=1,2,3,4$ and $t\ge 0$,
\begin{align}
 (\mbf E|\eta^x_{y_1}(t)|^{16+\kappa_1})^{\frac{1}{16+\kappa_1}}&\le C_1|y_1|e^{-C_2t},\label{eta1}\\
  (\mbf E|\eta^x_{y_1,y_2}(t)|^{8+\kappa_2})^{\frac{1}{8+\kappa_2}}&\le C_1(1+|x|^{q'})|y_1||y_2|e^{-C_2t},\label{eta2}\\
  (\mbf E|\eta^x_{y_1,y_2,y_3}(t)|^{4+\kappa_3})^{\frac{1}{4+\kappa_3}}&\le C_1(1+|x|^{2q'})|y_1||y_2||y_3|e^{-C_2t},\label{eta3}\\
  (\mbf E|\eta^x_{y_1,y_2,y_3,y_4}(t)|^{2})^{\frac{1}{2}}&\le C_1(1+|x|^{3q'})|y_1||y_2||y_3||y_4|e^{-C_2t}\label{eta4}.
\end{align}
\end{lem}
\begin{proof}
Similarly to \cite[Section 1.3.3]{meandiff01},  $\eta^x_{y_1}$ solves the following variational equation
\begin{align*}
	\begin{cases}
		\ud \eta^x_{y_1}(t)=\nabla b(X^x(t))\eta^x_{y_1}(t)\ud t+\nabla \sigma (X^x(t))\eta^x_{y_1}(t)\ud W(t),\\
		\eta^x_{y_1}(0)=y_1.
	\end{cases}
\end{align*} 
	Notice that for any $p\ge2$ and matrix $A$, it holds that $\nabla (|x|^p)=p|x|^{p-2}x$ and
\begin{align}\label{trace}
	\frac{1}{2}\text{trace}\big(\nabla ^2(|x|^p)AA^\top\big)\le\frac{1}{2}p(p-1)|x|^{p-2}\|A\|_{HS}^2.
\end{align}
For any $\kappa\in(0,1)$ and $\lambda>0$, by the It\^o formula, \eqref{trace}, $\sigma\in\mbf C_b^4(\mbb R^d)$ and \eqref{couple2},
\begin{align*}
	&\;\mbf E\big(e^{\lambda t}|\eta^x_{y_1}(t)|^{16+\kappa}\big)\\
	\le&\; |y_1|^{16+\kappa}+\lambda\int_0^te^{\lambda s}|\eta_{y_1}^x(s)|^{16+\kappa}\ud s\\
	&\;+\frac{1}{2}(16+\kappa)\mbf E\int_0^te^{\lambda s}|\eta^x_{y_1}(s)|^{14+\kappa}\Big[2\LL \eta^x_{y_1}(s),\nabla b(X^x(s))\eta^x_{y_1}(s)\RR	+(15+\kappa)\|\nabla \sigma(X^x(s))\eta^x_{y_1}(s)\|^2_{HS}\Big]\\
\le &\;|y_1|^{16+\kappa}+\big[\lambda+(8+\frac{\kappa}{2})(-L_5+\kappa L^2_\sigma)\big]\int_0^t\mbf E|\eta^x_{y_1}(s)|^{16+\kappa}\ud s,
\end{align*}
where $L_\sigma:=\sup\limits_{x\in\mbb R^d}\|\nabla \sigma(x)\|_{\otimes}$.
Letting $\kappa_1<L_5/L^2_\sigma$, $\lambda_1$ small enough, we obtain
\begin{align*}
	\mbf E|\eta^x_{y_1}(t)|^{16+\kappa_1}\le |y_1|^{16+\kappa_1}e^{-\lambda_1 t}\quad \forall ~t\in[0,T],
\end{align*}
which yields the \eqref{eta1}.

Secondly, similar to the argument for $\eta^x_{y_1}$, we have 
\begin{align*}
	\left\{\begin{array}{ll}
		\ud \eta^x_{y_1,y_2}(t)=&\nabla b(X^x(t))\eta^x_{y_1,y_2}(t)\ud t+\nabla^2 b(X^x(t))(\eta^x_{y_1}(t),\eta^x_{y_2}(t))\ud t\\
		&+\nabla \sigma (X^x(t))\eta^x_{y_1,y_2}(t)\ud W(t)+\nabla^2 \sigma(X^x(t))(\eta^x_{y_1}(t),\eta^x_{y_2}(t))\ud W(t),\\
		\eta^x_{y_1,y_2}(0)=&0.
	\end{array}
\right.
\end{align*} 
For any $\kappa,\lambda,\varepsilon_0\in(0,1)$, again by the It\^o formula, \eqref{trace}, $\sigma\in\mbf C_b^4(\mbb R^d)$ and the elementary  inequality $(a+b)^2\le (1+\varepsilon_0)a^2+(1+1/\varepsilon_0)b^2$ with $a,b\ge 0$, it holds that
\begin{align*}
	&\;\mbf E\big(e^{\lambda t}|\eta^x_{y_1,y_2}(t)|^{8+\kappa}\big)\\
	\le&\;\lambda\mbf E\int_0^t e^{\lambda s}|\eta^x_{y_1,y_2}(s)|^{8+\kappa}\ud s+(8+\kappa)\mbf E\int_0^te^{\lambda s}|\eta^x_{y_1,y_2}(s)|^{6+\kappa}\LL \eta^x_{y_1,y_2}(s),\nabla b(X^x(s))\eta^x_{y_1,y_2}(s)\RR\ud s\\
	&\;+(8+\kappa)\mbf E\int_0^te^{\lambda s}|\eta^x_{y_1,y_2}(s)|^{6+\kappa}\LL \eta^x_{y_1,y_2}(s),\nabla^2 b(X^x(s))(\eta^x_{y_1}(s),\eta^x_{y_2}(s))\RR\ud s\\
	&\;+\frac{1}{2}(8+\kappa)(7+\kappa)\mbf E\int_0^te^{\lambda s}|\eta^x_{y_1,y_2}(s)|^{6+\kappa}\|\nabla \sigma (X^x(s))\eta^x_{y_1,y_2}(s)+\nabla^2 \sigma(X^x(s))(\eta^x_{y_1}(s),\eta^x_{y_2}(s))\|^2_{HS}\ud s\\
	\le &\;\lambda\mbf E\int_0^t e^{\lambda s}|\eta^x_{y_1,y_2}(s)|^{8+\kappa}\ud s+\frac{1}{2}(8+\kappa)\mbf E\int_0^t	e^{\lambda s}|\eta^x_{y_1,y_2}(s)|^{6+\kappa}\Big[2\LL \eta^x_{y_1,y_2}(s),\nabla b(X^x(s))\eta^x_{y_1,y_2}(s)\RR\\
	&\;\qquad\qquad\qquad+(7+\kappa)(1+\varepsilon_0)\|\nabla \sigma (X^x(s))\eta^x_{y_1,y_2}(s)\|_{HS}^2\Big]\ud s\\
	&\;+K(\kappa)\mbf E\int_0^te^{\lambda s}|\eta^x_{y_1,y_2}(s)|^{7+\kappa}\|\nabla^2 b(X^x(s))\|_{\otimes}|\eta^x_{y_1}(s)||\eta^x_{y_2}(s)|\ud s\\
	&\;+K(\kappa,\varepsilon_0)\mbf E\int_0^te^{\lambda s}|\eta^x_{y_1,y_2}(s)|^{6+\kappa}|\eta^x_{y_1}(s)|^2|\eta^x_{y_2}(s)|^2\ud s.
\end{align*}
Further, taking  $\varepsilon_0\ll 1$ and  using \eqref{couple2}, we get
\begin{align}
	&\;\mbf E\big(e^{\lambda t}|\eta^x_{y_1,y_2}(t)|^{8+\kappa}\big)\nonumber\\
	\le&\; \big(\lambda-(4+\frac{\kappa}{2})L_5\big)\mbf E\int_0^te^{\lambda s}|\eta^x_{y_1,y_2}(s)|^{8+\kappa}\ud s+K(\kappa)\mbf E\int_0^te^{\lambda s}|\eta^x_{y_1,y_2}(s)|^{7+\kappa}\|\nabla^2 b(X^x(s))\|_{\otimes}|\eta^x_{y_1}(s)||\eta^x_{y_2}(s)|\ud s\nonumber\\
	&\;+K(\kappa)\mbf E\int_0^te^{\lambda s}|\eta^x_{y_1,y_2}(s)|^{6+\kappa}|\eta^x_{y_1}(s)|^2|\eta^x_{y_2}(s)|^2\ud s.\label{eq3}
\end{align}
It follows from the Young inequality $ab\le \varepsilon a^p+K(\varepsilon)b^q$ with $a,b\ge 0$, $\frac{1}{p}+\frac{1}{q}=1$, $p,q>1$ and  the H\"older inequality that for any $\varepsilon,\varepsilon'>0$,
\begin{align*}
&\;\mbf E	\big(|\eta^x_{y_1,y_2}(s)|^{7+\kappa}\|\nabla^2 b(X^x(s))\|_{\otimes}|\eta^x_{y_1}(s)||\eta^x_{y_2}(s)|\big)\nonumber\\
\le&\; \varepsilon \mbf E|\eta^x_{y_1,y_2}(s)|^{8+\kappa}+K(\varepsilon)\mbf E\Big[\big(\|\nabla^2 b(X^x(s))\|_{\otimes}|\eta^x_{y_1}(s)||\eta^x_{y_2}(s)|\big)^{8+\kappa}\Big]\nonumber\\
\le&\;\varepsilon \mbf E|\eta^x_{y_1,y_2}(s)|^{8+\kappa}+K(\varepsilon) \big(\mbf E|\eta^x_{y_1}(s)|^{(8+\kappa)(2+\varepsilon')}\big)^{\frac{1}{2+\varepsilon'}}\big(\mbf E|\eta^x_{y_2}(s)|^{(8+\kappa)(2+\varepsilon')}\big)^{\frac{1}{2+\varepsilon'}}\nonumber\\
&\;\cdot\big(\mbf E\|\nabla^2 b(X^x(s))\|_{\otimes}^{(8+\kappa)(1+2/\varepsilon')}\big)^{\frac{\varepsilon'}{2+\varepsilon'}}.
\end{align*} 
Taking sufficiently small $\kappa$ and $\varepsilon'$, from  Assumption \ref{assum3}, Proposition \ref{Xtesti}(1) and \eqref{eta1} it follows that for any $\varepsilon>0$,
\begin{align}
	&\;\mbf E	\big(|\eta^x_{y_1,y_2}(s)|^{7+\kappa}\|\nabla^2 b(X^x(s))\|_{\otimes}|\eta^x_{y_1}(s)||\eta^x_{y_2}(s)|\big)\nonumber\\
	\le&\; \varepsilon
	 \mbf E|\eta^x_{y_1,y_2}(s)|^{8+\kappa}+K(\varepsilon)(1+|x|^{(8+\kappa)q'})|y_1|^{8+\kappa}|y_2|^{8+\kappa}e^{-Ks}.\label{eq4}
\end{align}
Similarity, for any $\varepsilon>0$,
\begin{align}
	&\;\mbf E\big(|\eta^x_{y_1,y_2}(s)|^{6+\kappa}|\eta^x_{y_1}(s)|^2|\eta^x_{y_2}(s)|^2\big) \nonumber\\
	\le&\;\varepsilon \mbf E|\eta^x_{y_1,y_2}(s)|^{8+\kappa}+K(\varepsilon)\big(\mbf E|\eta^x_{y_1}(s)|^{2(8+\kappa)}\big)^{\frac{1}{2}}\big(\mbf E|\eta^x_{y_2}(s)|^{2(8+\kappa)}\big)^{\frac{1}{2}}\nonumber\\
	\le&\;\varepsilon \mbf E|\eta^x_{y_1,y_2}(s)|^{8+\kappa}+K(\varepsilon)|y_1|^{8+\kappa}|y_2|^{8+\kappa}e^{-Ks}. \label{eq5}
\end{align}
Plugging \eqref{eq4}-\eqref{eq5} into \eqref{eq3},  and taking sufficiently small $\kappa_2$, $\lambda_2$ and $\varepsilon$, one has
\begin{align*}
	\mbf E\big(e^{\lambda_2t}|\eta^x_{y_1,y_2}(t)|^{8+\kappa_2}\big)\le -K\mbf E\int_0^te^{\lambda_2s}|\eta^x_{y_1,y_2}(s)|^{8+\kappa_2}\ud s+K(1+|x|^{(8+\kappa_2)q'})|y_1|^{8+\kappa}|y_2|^{8+\kappa},
\end{align*} 
which produces \eqref{eta2}.

Further, $\eta^{x}_{y_1,y_2,y_3}$ solves the following SDE
 \begin{align*}
 	\begin{cases}
 		\ud \eta^x_{y_1,y_2,y_3}(t)=&\nabla b(X^x(t))\eta^x_{y_1,y_2,y_3}(t)\ud t+\nabla^2 b(X^x(t))(\eta^x_{y_1}(t),\eta^x_{y_2,y_3}(t))\ud t\\
 		&+\nabla^2 b(X^x(t))(\eta^x_{y_2}(t),\eta^x_{y_1,y_3}(t))\ud t+\nabla^2 b(X^x(t))(\eta^x_{y_3}(t),\eta^x_{y_1,y_2}(t))\ud t\\
 		&+\nabla^3 b(X^x(t))(\eta^x_{y_1}(t),\eta^x_{y_2}(t),\eta^x_{y_3}(t))\ud t+
 		\nabla \sigma(X^x(t))\eta^x_{y_1,y_2,y_3}(t)\ud W(t)\\
 		&+\nabla^2 \sigma(X^x(t))(\eta^x_{y_1}(t),\eta^x_{y_2,y_3}(t))\ud W(t)
 		+\nabla^2 \sigma(X^x(t))(\eta^x_{y_2}(t),\eta^x_{y_1,y_3}(t))\ud W(t)\\
 		&+\nabla^2 \sigma(X^x(t))(\eta^x_{y_3}(t),\eta^x_{y_1,y_2}(t))\ud W(t)
 		+\nabla^3\sigma (X^x(t))(\eta^x_{y_1}(t),\eta^x_{y_2}(t),\eta^x_{y_3}(t))\ud W(t),\\
 		\eta^x_{y_1,y_2,y_3}(0)=&0.
 	\end{cases}
 \end{align*}  
By the same argument for deriving \eqref{eq4}, using It\^o formula, \eqref{couple2} and $\sigma\in\mbf C^4_b(\mbb R^d)$, we have that for any $\kappa,\lambda\in(0,1)$,
\begin{align}
	&\;\mbf E\big(e^{\lambda t}|\eta^x_{y_1,y_2,y_3}(t)|^{4+\kappa}\big)\nonumber\\
	\le&\;\big(\lambda-(2+\frac{\kappa}{2})L_5\big)\mbf E\int_0^te^{\lambda s}|\eta^x_{y_1,y_2,y_3}(s)|^{4+\kappa}\ud s\nonumber\\
	+&\;K(\kappa)\mbf E\int_0^te^{\lambda s}|\eta^x_{y_1,y_2,y_3}(s)|^{3+\kappa}\|\nabla^2 b(X^x(s))\|_{\otimes}\big(|\eta^x_{y_1}(s)||\eta^x_{y_2,y_3}(s)|+|\eta^x_{y_2}(s)||\eta^x_{y_1,y_3}(s)|+|\eta^x_{y_3}(s)||\eta^x_{y_1,y_2}(s)|\big)\ud s\nonumber\\
	+&\;K(\kappa)\mbf E\int_0^te^{\lambda s}|\eta^x_{y_1,y_2,y_3}(s)|^{3+\kappa}\|\nabla^3 b(X^x(s))\|_{\otimes}|\eta^x_{y_1}(s)||\eta^x_{y_2}(s)||\eta^x_{y_3}(s)|\ud s\nonumber\\
	+&\;K(\kappa)\mbf E\int_0^te^{\lambda s}|\eta^x_{y_1,y_2,y_3}(s)|^{2+\kappa}\Big(|\eta^x_{y_1}(s)|^2|\eta^x_{y_2,y_3}(s)|^2+|\eta^x_{y_2}(s)|^2|\eta^x_{y_1,y_3}(s)|^2\nonumber\\
	&\;\qquad\qquad\qquad+|\eta^x_{y_3}(s)|^2|\eta^x_{y_1,y_3}(s)|^2+|\eta^x_{y_1}(s)|^2|\eta^x_{y_2}(s)|^2|\eta^x_{y_3}(s)|^2\Big)\ud s.\nonumber\\
	=:&\;\big(\lambda-(2+\frac{\kappa}{2})L_5\big)\mbf E\int_0^te^{\lambda s}|\eta^x_{y_1,y_2,y_3}(s)|^{4+\kappa}\ud s+I_1(t)+I_2(t)+I_3(t).\label{eq6}
\end{align}
It follows from the Young inequality, H\"older inequality, \eqref{eta1}-\eqref{eta2}, Assumption \ref{assum3} and Proposition \ref{Xtesti}(1)  that for sufficiently small $\kappa,\varepsilon,\varepsilon'$,
\begin{align*}
	&\;\mbf E\big(|\eta^x_{y_1,y_2,y_3}(s)|^{3+\kappa}\|\nabla^2 b(X^x(s))\|_{\otimes}|\eta^x_{y_{\chi(1)}}(s)||\eta^x_{y_{\chi(2)},y_{\chi(3)}}(s)|\big)\\
	\le &\;\varepsilon\mbf E|\eta^x_{y_1,y_2,y_3}(s)|^{4+\kappa}+K(\varepsilon)\big(\mbf E|\eta^x_{y_{\chi(1)}}(s)|^{(4+\kappa)(2+\varepsilon')}\big)^{\frac{1}{2+\varepsilon'}}\big(\mbf E|\eta^x_{y_{\chi(2),\chi(3)}}(s)|^{(4+\kappa)(2+\varepsilon')}\big)^{\frac{1}{2+\varepsilon'}}\\
	&\;\cdot \big(\mbf E\|\nabla^2 b(X^x(s))\|^{(4+\kappa)(1+2/\varepsilon')}_{\otimes}\big)^{\frac{\varepsilon'}{2+\varepsilon'}}\\
	\le &\; \varepsilon\mbf E|\eta^x_{y_1,y_2,y_3}(s)|^{4+\kappa}+K(\varepsilon)(1+|x|^{2q'(4+\kappa)})\big(|y_1||y_2||y_3|\big)^{4+\kappa}e^{-Ks},
\end{align*}
where $(\chi(1),\chi(2),\chi(3))$ is any permutation of $(1,2,3)$. Thus, for $\kappa,\lambda,\varepsilon\ll1$,
\begin{align}\label{eq7}
	I_1(t)\le K(\kappa)\varepsilon\mbf E\int_0^te^{\lambda s}|\eta^x_{y_1,y_2,y_3}(s)|^{4+\kappa}\ud s+K(\kappa,\varepsilon)(1+|x|^{2q'(4+\kappa)})\big(|y_1||y_2||y_3|\big)^{4+\kappa}.
\end{align}
Similarly, it can be verified that for $\kappa,\lambda,\varepsilon\ll1$,
\begin{gather}
	I_2(t)\le K\varepsilon\mbf E\int_0^te^{\lambda s}|\eta^x_{y_1,y_2,y_3}(s)|^{4+\kappa}\ud s+K(\varepsilon)(1+|x|^{q'(4+\kappa)})\big(|y_1||y_2||y_3|\big)^{4+\kappa},\label{eq8}\\
	I_3(t)\le K\varepsilon\mbf E\int_0^te^{\lambda s}|\eta^x_{y_1,y_2,y_3}(s)|^{4+\kappa}\ud s+K(\varepsilon)(1+|x|^{q'(4+\kappa)})\big(|y_1||y_2||y_3|\big)^{4+\kappa}\label{eq9}.
\end{gather}
Plugging \eqref{eq7}-\eqref{eq9} into \eqref{eq6} yields \eqref{eta3}.
Finally, by means of an analogous proof for \eqref{eta3}, we obtain \eqref{eta4}. Thus, the proof is finished.
\end{proof}

\begin{lem}\label{varphi}
	Let Assumptions \ref{assum1}-\ref{assum3} hold and $h\in\mbf C_b^4(\mbb R^d)$. 
	Let $\varphi$ be the function defined by \eqref{eq10}.
Then, for any $x\in\mbb R^d$,
\begin{gather}	
	|\varphi(x)|\le K(1+|x|),\label{varphigrowth}\\
	\|\nabla^i \varphi(x)\|_{\otimes}\le K(1+|x|^{(i-1)q'}),\quad i=1,2,3,4.\label{varphidiff}
\end{gather}
Moreover, $\varphi$ is a solution to the Poisson equation
\begin{align}
	\mcal L \varphi=h-\pi(h). \label{poisson}
\end{align}
\end{lem}
\begin{proof}
By \eqref{pif}, $|\varphi(x)|\le K(h)(1+|x|)\int_0^\infty e^{-\lambda_1 t}\ud t\le K(h)(1+|x|)$, which indicates that $\varphi$ is well defined and \eqref{varphigrowth} holds.

Denoting $u(t,x):=\mbf Eh(X^{x}(t))$, we have that for any $x,y_1\in\mbb R^d$, 
$\nabla_x u(t,x)y_1=\mbf E(\nabla h(X^x(t))\eta^x_{y_1}(t))$, due to the definition of $\eta^x_{y_1}$ and $h\in\mbf C^1_b(\mbb R^d)$. It follows from \eqref{eta1}, $h\in\mbf C_b^4(\mbb R^d)$ and the H\"older inequality that
$|\nabla_x u(t,x)y_1|\le K\mbf E|\eta^x_{y_1}(t)|\le K|y_1|e^{-C_2t}$. By the arbitrariness of $y_1$, $|\nabla_x u(t,x)|\le Ke^{-C_2t}$, which implies 
$|\nabla \varphi(x)|\le \int_0^\infty|\nabla_x u(t,x)|\ud t\le K.$

Further,
$\nabla ^2_x u(t,x)(y_1,y_2)=\mbf E\big(\nabla h(X^x(t))\eta^x_{y_1,y_2}(t)+\nabla^2 h(X^x(t))(\eta^x_{y_1}(t),\eta^x_{y_2}(t))\big)$  for any $x,y_1,y_2\in\mbb R^d$. Then \eqref{eta2}, $h\in\mbf C_b^4(\mbb R^d)$ and the H\"older inequality yield
$$|\nabla ^2_x u(t,x)(y_1,y_2)|\le K\mbf E|\eta^x_{y_1,y_2}(t)|+K\mbf E|\eta^x_{y_1}(t)||\eta^x_{y_2}(t)|\le K(1+|x|^{q'})|y_1||y_2|e^{-C_2t}.$$
This gives $\|\nabla ^2_x u(t,x)\|_{HS}\le K(1+|x|^{q'})e^{-C_2t}$ and thus $\|\nabla ^2 \varphi(x)\|_{HS}\le K(1+|x|^{q'})$.
Similarly, it can be verified that \eqref{varphidiff} holds for $i=3,4$.

By the It\^o formula, 
$\mbf Eh(X^x(t))=h(x)+\int_0^t\mbf E\mcal Lh(X^x(s))\ud s$, which gives $\mbf E\mcal Lh(X^x(t))=\frac{\PD }{\PD t}\mbf Eh(X^x(t))$, i.e., $$\mcal Lu(t,x)=\frac{\PD }{\PD t}u(t,x).$$ It follows from \eqref{generator}, \eqref{bgrowth} and the previous estimates for $\nabla_x u(t,x)$ and $\nabla_x^2u(t,x)$ that
\begin{align}
	|\mcal Lu(t,x)|\le K(1+|x|^{q}+|x|^{q'})e^{-C_2t}.\label{Eq8}
\end{align}
Thus, we can exchange the operator $\mcal L$ and the integration in $t$  for $\mcal L\int_0^{\infty} (u(t,x)-\pi(h))\ud t$. Accordingly, using \eqref{pif} and \eqref{Eq8} yields that for any $x\in\mbb R^d$,
\begin{align*}
	\mcal L\varphi(x)&=-\int_0^{\infty}\mcal Lu(t,x)\ud t=-\int_0^{\infty}\frac{\PD }{\PD t}u(t,x)\ud t\\
	&=u(0,x)-\lim_{t\to+\infty}u(t,x)
	=h(x)-\lim_{t\to+\infty}\mbf Eh(X^x(t))=h(x)-\pi(h).
\end{align*}
This finishes the proof.
\end{proof}

\subsection{Detailed proof}
In this part, we give the proof of Theorem \ref{CLT2}. As is mentioned previously, we will split $\frac{1}{\sqrt{\tau}}(\Pi_{\tau,2}(h)-\pi(h))$ into a martingale difference series sum and a negligible remainder, based on the Poisson equation \eqref{poisson}. 

\textbf{Proof of Theorem \ref{CLT2}.} For the convenience of notations, we denotes $m=\tau^{-2}$ with $\tau$ being sufficiently small. By \eqref{poisson},  we have
\begin{align*}
	&\phantom{=}\frac{1}{\sqrt{\tau}}(\Pi_{\tau,2}(h)-\pi(h))\\
	&=\tau^{-\frac12}\frac{1}{m}\sum_{k=0}^{m-1}(h(\bar X^x_k)-\pi(h))=\tau^{\frac32}\sum_{k=0}^{m-1}\mcal L\varphi(\bar{X}^x_k)\\
	&=\tau^{\frac12}\sum_{k=0}^{m-1}\Big(\mcal L\varphi(\bar{X}^x_k)\tau-\big(\varphi(\bar{X}^x_{k+1})-\varphi(\bar{X}^x_k)\big)\Big)+\tau^{\frac12}(\varphi(\bar{X}^x_m)-\varphi(x)).
\end{align*}
Lemma \ref{varphi} enables us to apply the Taylor expansion for $\varphi$:
\begin{align*}
\varphi(\bar{X}^x_{k+1})-\varphi(\bar{X}^x_k)
	=&\;\LL\nabla\varphi(\bar{X}^x_k),\Delta \bar{X}^x_k\RR+\frac{1}{2}\LL\nabla^2\varphi(\bar {X}^x_k),\Delta \bar{X}^x_k(\Delta\bar{X}^x_k)^\top\RR_{HS}\\
	&\;+\frac{1}{2}\int_0^1(1-\theta)^2\nabla^3\varphi(\bar X_k^x+\theta \Delta \bar X^x_k)(\Delta \bar{X}^x_k,\Delta \bar{X}^x_k,\Delta \bar{X}^x_k)\ud\theta,
\end{align*} 
where $\Delta \bar{X}^x_k:=b(\bar X^x_{k+1})\tau+\sigma(\bar{X}^x_k)\Delta W_k$, $k=0,1,\ldots,m$.
It follows from \eqref{generator} and the above formulas that
\begin{align*}
	\frac{1}{\sqrt{\tau}}(\Pi_{\tau,2}(h)-\pi(h))=\mcal H_\tau+\mcal R_{\tau},
\end{align*}
where $\mcal H_{\tau}$ and $\mcal R_\tau$ are given by
\begin{gather}
	\mcal H_\tau:=-\tau^{\frac{1}{2}}\sum_{k=0}^{m-1}\LL\nabla\varphi(\bar{X}^x_k),\sigma(\bar{X}^x_k)\Delta W_k\RR,\quad
	\mcal R_\tau=\sum_{i=1}^6R_{\tau,i},\label{HRtau}
\end{gather}
with
\begin{align}
R_{\tau,1}:=&\;\tau^{\frac{1}{2}}\big(\varphi(\bar{X}^x_m)-\varphi(x)\big),\label{Rtau1}\\
R_{\tau,2}:=&\;-\tau^{\frac32}\sum_{k=0}^{m-1}\LL\nabla\varphi(\bar{X}^x_k),b(\bar{X}^x_{k+1})-b(\bar{X}^x_k)\RR,\label{Rtau2}\\
R_{\tau,3}:=&\;\frac{1}{2}	\tau^{\frac12}\sum_{k=0}^{m-1}\LL\nabla^2\varphi(\bar{X}^x_k),\sigma(\bar{X}^x_k)(\tau I_D-\Delta W_k\Delta W_k^\top)\sigma(\bar{X}^x_k)^\top\RR_{HS},\label{Rtau3}\\
R_{\tau,4}:=&\;-\frac{1}{2}	\tau^{\frac52}\sum_{k=0}^{m-1}\LL\nabla^2\varphi(\bar{X}^x_k),b(\bar{X}^x_{k+1})b(\bar{X}^x_{k+1})^\top\RR_{HS},\label{Rtau4}\\
R_{\tau,5}:=&\;-	\tau^{\frac32}\sum_{k=0}^{m-1}\LL\nabla^2\varphi(\bar{X}^x_k),b(\bar{X}^x_{k+1})(\sigma(\bar{X}^x_k)\Delta W_k)^\top\RR_{HS}\label{Rtau5},\\
R_{\tau,6}:=&\;-\frac{1}{2}	\tau^{\frac12}\sum_{k=0}^{m-1}\int_0^1(1-\theta)^2\nabla^3\varphi(\bar X^x_k+\theta \Delta \bar X^x_k)(\Delta \bar{X}^x_k,\Delta \bar{X}^x_k,\Delta \bar{X}^x_k)\ud\theta.\label{Rtau6}
\end{align}
By Lemmas \ref{Htaulimit}-\ref{Rtaulimit} below and the Slutsky theorem,  $\frac{1}{\sqrt{\tau}}(\Pi_{\tau,2}(h)-\pi(h))\overset{d}{\longrightarrow}\mcal N(0,\pi(|\sigma^\top\nabla\varphi|^2))$ as $\tau\to 0$ and the proof is complete. \hfill$\square$.

\begin{lem}\label{Htaulimit}
Suppose that Assumptions \ref{assum1}-\ref{assum3} hold. Then for any $x\in\mbb R^d$,
\begin{align*}
	\mcal H_\tau\overset{d}{\longrightarrow}\mcal N(0,\pi(|\sigma^\top\nabla\varphi|^2))\quad\text{as}~\tau\to 0.
\end{align*}
\end{lem}
\begin{proof}
Recall that 	$\mcal H_\tau:=-\tau^{\frac{1}{2}}\sum_{k=0}^{m-1}\LL\nabla\varphi(\bar{X}^x_k),\sigma(\bar{X}^x_k)\Delta W_k\RR$ with $m=\tau^{-2}$.
	According to \cite[Theorem 2.3]{CLT74}, it suffices to show that
	\begin{gather}
		\lim_{\tau\to0}\tau\mbf E\max_{0\le k\le m-1}|Z_k|^2=0,\label{eq11}\\
	\tau\sum_{k=0}^{m-1}|Z_k|^2\overset{\mbf P}{\longrightarrow}\pi(|\sigma^\top\nabla\varphi|^2)\quad{as}~\tau\to 0,\label{eq12}
	\end{gather}
where $Z_k:=\LL\nabla\varphi(\bar{X}^x_k),\sigma(\bar{X}^x_k)\Delta W_k\RR$, $k=0,1,\ldots,m$.
It follows from the boundedness of $\sigma$ and \eqref{varphidiff} that 
\begin{align*}
	&\;\tau\mbf E\max_{0\le k\le m-1}|Z_k|^2\\
	\le&\;\tau \mbf E\max_{0\le k\le m-1}\big(|Z_k|^2\mbf 1_{\{|Z_k|^2\le 1\}}\big)+\tau \mbf E\max_{0\le k\le m-1}\big(|Z_k|^2\mbf 1_{\{|Z_k|^2>1\}}\big)\\
	\le &\;\tau+\tau\sum_{k=0}^{m-1}\mbf E\big(|Z_k|^2\mbf 1_{\{|Z_k|^2>1\}}\big)
	\le \tau +\tau\sum_{k=0}^{m-1}\mbf E|Z_k|^4\\
	\le &\;\tau+K\tau\sum_{k=0}^{m-1}\mbf E|\Delta W_k|^4\le \tau +K\tau^3m\le K\tau,
\end{align*}
which implies \eqref{eq11}.

By \eqref{varphidiff}, for any $x,y\in\mbb R^d$,
\begin{align*}
|\nabla\varphi(x)-\nabla\varphi(y)|= \Big|\int_0^1\nabla^2\varphi(x+\theta(y-x))(y-x)\ud \theta\Big|\le K(1+|x|^{q'}+|y|^{q'})|x-y|,
\end{align*}
which together with the assumptions on $\sigma$ gives 
\begin{gather}
	|\sigma^\top\nabla\varphi|^2\in Poly(q'+1,\mbb R^d). \label{eq14}
\end{gather} As a result of \eqref{pitaupi}, $\big|\pi_\tau(|\sigma^\top\varphi|^2)-\pi(|\sigma^\top\varphi|^2)\big|\le K\tau^{\frac12}$.
Thus, once we show that
\begin{align}\label{eq13}
	\tau\sum_{k=0}^{m-1}|Z_k|^2-\pi_\tau(|\sigma^\top\nabla\varphi|^2)\overset{\mbf P}{\longrightarrow}0 \quad\text{as}~\tau\to 0,
\end{align}
we obtain \eqref{eq12} and complete the proof.

According to \eqref{pitau} and \eqref{eq14},
\begin{align*}
	\Big|\mbf E\big(|\sigma(\bar{X}_k^x)^\top\nabla\varphi(\bar{X}_k^x)|^2\big)-\pi_\tau(|\sigma^\top\nabla\varphi|^2)\Big|\le K(1+|x|^{q'+1})e^{-\xi_1k\tau}, ~k\ge 0.
\end{align*} 
By the above formula and the property of conditional expectations, for any $j\ge i$,
\begin{align}
	&\Big|\mbf E_i\big(|\sigma(\bar{X}^x_j)^\top\nabla\varphi(\bar{X}^x_j)|^2\big)-\pi_\tau(|\sigma^\top\nabla\varphi|^2)\Big|=\Big|\mbf E_i\big(|\sigma(\bar{X}_j^{i,\bar{X}^x_i})^\top\nabla\varphi(\bar{X}_j^{i,\bar{X}^x_i})|^2\big)-\pi_\tau(|\sigma^\top\nabla\varphi|^2)\Big|\nonumber\\
	=&\Big|\Big(\mbf E\big(|\sigma(\bar{X}_j^{i,y})^\top\nabla\varphi(\bar{X}_j^{i,y})|^2\big)-\pi_\tau(|\sigma^\top\varphi|^2)\Big)\big|_{y=\bar{X}^x_i}\Big|\nonumber\\
	\le &K(1+|\bar{X}^x_i|^{q'+1})e^{-\xi_1(j-i)\tau}.\label{eq16}
\end{align}
Hereafter, we denote by $\mbf E_i(\cdot)$ the conditional expectation $\mbf E(\cdot|\mcal F_{t_i})$, $i\ge 0$.
Further, 
\begin{align}
	&\;\mbf E\Big(\tau\sum_{k=0}^{m-1}|Z_k|^2-\pi_\tau(|\sigma^\top\nabla\varphi|^2\Big)^2=\mbf E\Big(\tau^2\sum_{k=0}^{m-1}\big(\tau^{-1}|Z_k|^2-\pi_\tau(|\sigma^\top\nabla\varphi|^2\big)\Big)^2\nonumber\\
	=&\;\tau^4\sum_{i=0}^{m-1}\mbf E\big(\tau^{-1}|Z_i|^2-\pi_\tau(|\sigma^\top\nabla\varphi|^2)\big)^2\nonumber\\
	&\;+2\tau^4\sum_{0\le i<j\le m-1}\mbf E\Big[\big(\tau^{-1}|Z_i|^2-\pi_\tau(|\sigma^\top\nabla\varphi|^2)\big)\big(\tau^{-1}|Z_j|^2-\pi_\tau(|\sigma^\top\nabla\varphi|^2)\big)\Big].\label{Eq1}
\end{align}
It follows from the boundedness of $\sigma$, \eqref{varphidiff}, Proposition \ref{Xtergodic}(2), \eqref{pitaupi}  and \eqref{eq14} that for $\tau\in(0,1)$ and $i\ge 0$,
\begin{align}\label{eq15}
	&\phantom{\le}\mbf E\big(\tau^{-1}|Z_i|^2-\pi_\tau(|\sigma^\top\nabla\varphi|^2)\big)^2\le 2\tau^{-2}\mbf E|Z_i|^4+2(\pi_\tau(|\sigma^\top\nabla\varphi|^2))^2\nonumber\\
	&\le K+4\big(\pi_\tau(|\sigma^\top\nabla\varphi|^2)-\pi(|\sigma^\top\nabla\varphi|^2)\big)^2+4\big(\pi(|\sigma^\top\nabla\varphi|^2)\big)^2\nonumber\\
	&\le K+K\tau+K(\pi(|\cdot|^{q'+1}))^2\le K.
\end{align}

By the property of conditional expectations,
$$\mbf E_j|Z_j|^2=\big(\mbf E\LL x,\Delta W_j\RR^2\big)\big|_{x=\sigma(\bar{X}^x_j)^\top\nabla\varphi(\bar{X}^x_j)}=\tau |\sigma(\bar{X}^x_j)^\top\nabla\varphi(\bar{X}^x_j)|^2.$$
Thus, $\mbf E_{i+1}|Z_j|^2=\mbf E_{i+1}(\mbf E_j|Z_j|^2)=\tau \mbf E_{i+1}|\sigma(\bar{X}^x_j)^\top\nabla\varphi(\bar{X}^x_j)|^2$ for any $j>i$.
Combining the above relation, \eqref{eq16}, \eqref{eq15} and Theorem \ref{pth-moment}, we have that for  $j>i$ and $\tau<1$,
\begin{align}
	&\;\Big|\mbf E\Big[\big(\tau^{-1}|Z_i|^2-\pi_\tau(|\sigma^\top\nabla\varphi|^2)\big)\big(\tau^{-1}|Z_j|^2-\pi_\tau(|\sigma^\top\nabla\varphi|^2)\big)\Big]\Big|\nonumber\\
	=&\;\Big|\mbf E\Big[\big(\tau^{-1}|Z_i|^2-\pi_\tau(|\sigma^\top\nabla\varphi|^2)\big)\big(\tau^{-1}\mbf E_{i+1}|Z_j|^2-\pi_\tau(|\sigma^\top\nabla\varphi|^2)\big)\Big]\Big|\nonumber\\
	\le&\; \mbf E\Big[\big|\tau^{-1}|Z_i|^2-\pi_\tau(|\sigma^\top\nabla\varphi|^2)\big|\big|\mbf E_{i+1}\big(|\sigma(\bar{X}^x_j)^\top\nabla\varphi(\bar{X}^x_j)|^2\big)-\pi_\tau(|\sigma^\top\nabla\varphi|^2)\big|\Big]\nonumber\\
	\le &\;Ke^{-\xi_1(j-i-1)\tau}\mbf E\Big[\big|\tau^{-1}|Z_i|^2-\pi_\tau(|\sigma^\top\nabla\varphi|^2)\big|(1+|\bar{X}^x_{i+1}|^{q'+1})\Big]\nonumber\\
	\le &\; Ke^{-\xi_1(j-i)\tau}\big(\mbf E\big|\tau^{-1}|Z_i|^2-\pi_\tau(|\sigma^\top\nabla\varphi|^2)\big|^2\big)^{\frac12}(1+ \big(\mbf E|\bar{X}^x_{i+1}|^{2q'+2}\big)^{\frac12})\nonumber\\
	\le &\;K(x)e^{-\xi_1(j-i)\tau}.\label{Eq2}
\end{align}
Plugging \eqref{eq15}-\eqref{Eq2} into \eqref{Eq1} yields
\begin{align*}
&\;\mbf E\Big(\tau\sum_{k=0}^{m-1}|Z_k|^2-\pi_\tau(|\sigma^\top\nabla\varphi|^2\Big)^2\\
\le&\; K\tau^4m+K(x)\tau^4\sum_{0\le i<j\le m-1}	e^{-\xi_1(j-i)\tau}\\
=&\; K\tau^2+K(x)\tau^4\sum_{i=0}^{m-1}\sum_{j=i+1}^{m-1}e^{-\xi_1(j-i)\tau}\\
\le &\;K\tau^2+K(x)\tau^4m\sum_{j=1}^{\infty}e^{-\xi_1j\tau}\le K(x)\tau\to0\quad\text{as}~\tau\to 0,
\end{align*}
which leads to \eqref{eq13} and finishes the proof.
\end{proof}

\begin{lem}\label{Rtaulimit}
Suppose that Assumptions \ref{assum1}-\ref{assum3} hold. Then for any $x\in\mbb R^d$, $\mcal R_\tau\overset{\mbf P}{\longrightarrow}0$  as $\tau$ tends to $0$.
\end{lem}
\begin{proof}
We will  prove $\lim\limits_{\tau\to0}\mbf E|\mcal R_\tau|=0$ to obtain the conclusion.

\emph{Estimate of $\mcal R_{\tau,1}$.} By Theorem \ref{pth-moment}, \eqref{varphigrowth} and \eqref{Rtau1},
$\mbf E|\mcal R_{\tau,1}|\le K\tau^{\frac12}(1+\sup\limits_{n\ge 0}\mbf E|\bar{X}^x_n|)\le K(x)\tau^{\frac12}.$

\emph{Estimate of $\mcal R_{\tau,2}$.} By means of \eqref{bgrowth}, Assumption \ref{assum3}, Theorem \ref{pth-moment}, \eqref{varphidiff}  and the H\"older inequality, we have that for any $p\ge 1$, $i=2,3,4$ and $j=1,2$,
\begin{align}
\sup_{k\ge0}\mbf E|b(\bar X^x_k)|^p&\le K(1+\sup_{k\ge0}\mbf E|\bar{X}^x_k|^{pq})\le K(1+|x|^{pq}), \label{Eq5}\\
\sup_{k\ge0}\mbf E\|\nabla^j b(\bar X^x_k)\|_{\otimes}^p&\le K(1+\sup_{k\ge0}\mbf E|\bar{X}^x_k|^{pq'})\le K(1+|x|^{pq'}), \label{Eq5'}\\
 \sup_{k\ge0}\mbf E\|\nabla^i\varphi(\bar X^x_k)\|_{\otimes}^p&\le K(1+\sup_{k\ge0}\mbf E|\bar{X}^x_k|^{(i-1)pq'})\le K(1+|x|^{(i-1)pq'}). \label{Eq4}
\end{align} 
Noting that 
\begin{align*}
	b(\bar{X}^x_{k+1})-b(\bar{X}^x_k)=\nabla b(\bar{X}^x_k)\Delta \bar X^x_k+\int_0^1(1-\theta)\nabla^2 b(\bar{X}^x_k+\theta\Delta \bar X^x_k)(\Delta \bar X^x_k,\Delta \bar X^x_k)\ud \theta,
\end{align*}
one obtains from \eqref{Rtau2} that
\begin{align*}
	\mcal R_{\tau,2}=&\;-\tau^{\frac32}\sum_{k=0}^{m-1}\LL\nabla\varphi(\bar X^x_k),\nabla b(\bar X_{k})\sigma(\bar X^x_k)\Delta W_k\RR\\
	&\;-\tau^{\frac52}\sum_{k=0}^{m-1}\LL\nabla\varphi(\bar X^x_k),\nabla b(\bar X^x_{k})b(\bar X^x_{k+1})\RR\\
	&\;-\tau^{\frac32}\sum_{k=0}^{m-1}\int_0^1(1-\theta)\LL\nabla\varphi(\bar X^x_k),\nabla^2b(\bar{X}^x_k+\theta\Delta \bar X^x_k)(\Delta \bar X^x_k,\Delta \bar X^x_k)\RR\ud \theta\\
	=:&\mcal R_{\tau,2}^1+\mcal R_{\tau,2}^2+\mcal R_{\tau,2}^3.
\end{align*}
By the property of conditional expectations, for $i<j$,
\begin{align*}
	&\;\mbf E\big[\LL\nabla\varphi(\bar X^x_i),\nabla b(\bar X^x_{i})\sigma(\bar X^x_i)\Delta W_i\RR\LL\nabla\varphi(\bar X^x_j),\nabla b(\bar X^x_{j})\sigma(\bar X^x_j)\Delta W_j\RR\big]\\
	=&\;\mbf E\big[\LL\nabla\varphi(\bar X^x_i),\nabla b(\bar X^x_{i})\sigma(\bar X^x_i)\Delta W_i\RR\LL\nabla\varphi(\bar X^x_j),\nabla b(\bar X^x_{j})\sigma(\bar X^x_j)\mbf E_j(\Delta W_j)\RR\big]=0.
\end{align*}
The above relation, combined with the boundedness of $\sigma$, \eqref{varphidiff} and \eqref{Eq5'}, gives
\begin{align*}
	\mbf E|\mcal R_{\tau,2}^1|^2&=\tau^3\sum_{k=0}^{m-1}\mbf E\LL\nabla\varphi(\bar X^x_k),\nabla b(\bar X^x_{k})\sigma(\bar X^x_k)\Delta W_k\RR^2\\
	&\le K\tau^4\sum_{k=0}^{m-1}\mbf E|\nabla b(\bar X^x_{k})|^2\le K(x)\tau ^2.
\end{align*}
Applying the H\"older inequality, \eqref{varphidiff} and \eqref{Eq5}-\eqref{Eq5'}, we have
\begin{align*}
	\mbf E|\mcal R_{\tau,2}^2|\le K\tau^{\frac52}\sum_{k=0}^{m-1}(\mbf E|\nabla b(\bar{X}^x_k)|^2)^{\frac12}(\mbf E| b(\bar{X}^x_{k+1})|^2)^{\frac12}\le K(x)\tau^{\frac12}.
\end{align*}
Further, for any $p\ge 1$ and $k\ge 0$, it follows from the Minkowski inequality, \eqref{Eq5} and the boundedness of $\sigma$  that 
\begin{align}
	(\mbf E|\Delta \bar{X}^x_{k}|^p)^{\frac1p}\le \tau(\mbf E|b(\bar{X}^x_{k+1})|^{p})^{\frac 1p}+K(\mbf E|\Delta W_k|^p)^{\frac{1}{p}}\le K(1+|x|^q)\tau^{\frac12}.\label{eq17}
\end{align}
This together with the H\"older inequality, Assumption \ref{assum3} and Theorem \ref{pth-moment} yields
\begin{align*}
		\mbf E|\mcal R_{\tau,2}^3|\le K\tau^{\frac32}\sum_{k=0}^{m-1}(\mbf E|\Delta \bar X^x_k|^4)^{\frac{1}{2}}\big(1+(\mbf E|\Delta \bar X^x_k|^{2q'})^{\frac{1}{2}}+(\mbf E|\bar X^x_k|^{2q'})^{\frac{1}{2}}\big)\le  K(x)\tau^{\frac12}.
\end{align*}
In this way, we get $
	\mbf E|\mcal R_{\tau,2}|\le 	(\mbf E|\mcal R_{\tau,2}^1|^2)^{\frac12}+	\mbf E|\mcal R_{\tau,2}^2|+	\mbf E|\mcal R_{\tau,2}^3|\le K(x)\tau^{\frac12}.$

\emph{Estimate of $\mcal R_{\tau,3}$.} Notice that that for $i<j$,
\begin{align}
	&\;\mbf E\Big[\LL\nabla^2\varphi(\bar{X}^x_i),\sigma(\bar{X}^x_i)(\tau I_D-\Delta W_i\Delta W_i^\top)\sigma(\bar{X}^x_i)^\top\RR_{HS}\nonumber\\
	&\quad\cdot\LL\nabla^2\varphi(\bar{X}^x_j),\sigma(\bar{X}^x_j)(\tau I_D-\Delta W_j\Delta W_j^\top)\sigma(\bar{X}^x_j)^\top\RR_{HS}\Big]\nonumber\\
	=&\;\mbf E\Big[\LL\nabla^2\varphi(\bar{X}^x_i),\sigma(\bar{X}^x_i)(\tau I_D-\Delta W_i\Delta W_i^\top)\sigma(\bar{X}^x_i)^\top\RR_{HS}\nonumber\\
	&\quad\cdot\LL\nabla^2\varphi(\bar{X}^x_j),\sigma(\bar{X}^x_j)\mbf E_j(\tau I_D-\Delta W_j\Delta W_j^\top)\sigma(\bar{X}^x_j)^\top\RR_{HS}\Big]=0.\label{Eq3}
\end{align}
Combining \eqref{Rtau3}, \eqref{Eq3},    the boundedness of $\sigma$ and \eqref{Eq4}, we arrive at
\begin{align*}
	\mbf E|\mcal R_{\tau,3}|^2&=\frac{\tau}{4}\sum_{k=0}^{m-1}\mbf E\LL\nabla^2\varphi(\bar{X}^x_k),\sigma(\bar{X}^x_k)(\tau I_D-\Delta W_k\Delta W_k^\top)\sigma(\bar{X}^x_k)^\top\RR_{HS}^2\\
&\le K\tau \sum_{k=0}^{m-1}\mbf E\big(\|\nabla^2\varphi(\bar{X}^x_k)\|^2_{HS}(\tau^2+|\Delta W_k|^4)\big) \\
&\le K\tau\sum_{k=0}^{m-1}\big(\mbf E\|\nabla^2\varphi(\bar{X}^x_k)\|^4_{HS}\big)^{\frac12}(\tau^2+(\mbf E|\Delta W_k|^8)^{\frac12})\le K(x)\tau.
\end{align*}

\emph{Estimate of $\mcal R_{\tau,4}$.} 
By  \eqref{Rtau4}, \eqref{Eq5}, \eqref{Eq4} and the H\"older inequality,
\begin{align*}
	\mbf E|R_{\tau,4}|&\le  K\tau^{\frac52}\sum_{k=0}^{m-1}(\mbf E|\nabla^2\varphi(\bar{X}^x_k)|^2)^{\frac12}(\mbf E|b(\bar{X}^x_{k+1})|^4)^{\frac12}\le K(x)\tau^{\frac52}m\le K(x)\tau.
\end{align*}

\emph{Estimate of $\mcal R_{\tau,5}$.} We decompose $\mcal R_{\tau,5}$ (see \eqref{Rtau5}) into $\mcal R_{\tau,5}=\mcal R_{\tau,5}^1+\mcal R_{\tau,5}^2$ with
\begin{align*}
	\mcal R_{\tau,5}^1&:=-	\tau^{\frac32}\sum_{k=0}^{m-1}\LL\nabla^2\varphi(\bar{X}^x_k),(b(\bar{X}^x_{k+1})-b(\bar X^x_k))(\sigma(\bar{X}^x_k)\Delta W_k)^\top\RR_{HS},\\
	\mcal R_{\tau,5}^2&:=-	\tau^{\frac32}\sum_{k=0}^{m-1}\LL\nabla^2\varphi(\bar{X}^x_k),b(\bar{X}^x_{k})(\sigma(\bar{X}^x_k)\Delta W_k)^\top\RR_{HS}.
\end{align*}
By   the H\"older inequality, \eqref{Eq4},  \eqref{bsuper}, Theorem \ref{pth-moment} and \eqref{eq17},
\begin{align*}
	\mbf E|\mcal R_{\tau,5}^1|&\le K\tau^{\frac32}m\sup_{k\ge0}\big(\mbf E|\nabla^2\varphi(\bar{X}^x_k)|^3\big)^{\frac{1}{3}}\big(\mbf E|\Delta W_k|^3\big)^{\frac{1}{3}}\big(\mbf E|b(\bar{X}^x_{k+1})-b(\bar X^x_k)|^3\big)^{\frac{1}{3}}\\
&\le K(x)(1+\sup_{k\ge0}(\mbf E|\bar{X}^x_k|^{6q-6})^\frac{1}{6})\sup_{k\ge0}\big(\mbf E|\Delta \bar{X}^x_{k}|^6\big)^{\frac16}	\\
&\le K(x)\tau^{\frac12}.
\end{align*}

Similar to \eqref{Eq3}, one has that for $i<j$,
\begin{align*}
	\mbf E \Big[&\LL\nabla^2\varphi(\bar{X}^x_i),b(\bar{X}^x_{i})(\sigma(\bar{X}^x_i)\Delta W_i)^\top\RR_{HS}\cdot\LL\nabla^2\varphi(\bar{X}^x_j),b(\bar{X}^x_{j})(\sigma(\bar{X}^x_j)\Delta W_j)^\top\RR_{HS}\Big]=0.
\end{align*}
The above formula, combined with \eqref{Eq5}, \eqref{Eq4} and the H\"older inequality, yields
\begin{align*}
	\mbf E|\mcal R_{\tau,5}^2|^2&=\tau^3\sum_{k=0}^{m-1}\mbf E\LL\nabla^2\varphi(\bar{X}^x_k),b(\bar{X}^x_{k})(\sigma(\bar{X}^x_k)\Delta W_k)^\top\RR_{HS}^2\\
	&\le K\tau^3\sum_{k=0}^{m-1}\big(\mbf E|\nabla^2\varphi(\bar{X}^x_k)|^6\big)^{\frac13}\big(\mbf E|b(\bar{X}^x_k)|^6\big)^{\frac13}\big(\mbf E|\Delta W_k|^6\big)^{\frac13}\\
	&\le K(x)\tau^2.
\end{align*}
Thus, $\mbf  E|\mcal R_{\tau,5}|\le \mbf  E|\mcal R^1_{\tau,5}|+\big(\mbf  E|\mcal R^2_{\tau,5}|^2\big)^{\frac12}\le K(x)\tau^{\frac{1}{2}}$.
\end{proof}

\emph{Estimate of $\mcal R_{\tau,6}$.}
Plugging $\Delta \bar X^x_k=b(\bar X^x_{k+1})\tau+\sigma(\bar X^x_{k})\Delta W_k$ into \eqref{Rtau6} gives $\mcal R_{\tau,6}=\sum_{i=1}^4\mcal R_{\tau,6}^i$ with
\begin{align*}
\mcal R_{\tau,6}^1&:=-\frac{\tau^{\frac72}}{2}\sum_{k=0}^{m-1}\int_0^1(1-\theta)^2\nabla^3\varphi(\bar{X}^x_k+\theta\Delta \bar{X}^x_k)\big(b(\bar{X}^x_{k+1}),b(\bar{X}^x_{k+1}),b(\bar{X}^x_{k+1})\big)\ud \theta,\\
\mcal R_{\tau,6}^2&:=-\frac{3\tau^{\frac52}}{2}\sum_{k=0}^{m-1}\int_0^1(1-\theta)^2\nabla^3\varphi(\bar{X}^x_k+\theta\Delta \bar{X}^x_k)\big(b(\bar{X}^x_{k+1}),b(\bar{X}^x_{k+1}),\sigma(\bar{X}^x_k)\Delta W_k\big)\ud \theta,\\
R_{\tau,6}^3&:=-\frac{3\tau^{\frac32}}{2}\sum_{k=0}^{m-1}\int_0^1(1-\theta)^2\nabla^3\varphi(\bar{X}^x_k+\theta\Delta \bar{X}^x_k)\big(b(\bar{X}^x_{k+1}),\sigma(\bar{X}^x_k)\Delta W_k,\sigma(\bar{X}^x_k)\Delta W_k\big)\ud \theta,\\
R_{\tau,6}^4&:=-\frac{\tau^{\frac12}}{2}\sum_{k=0}^{m-1}\int_0^1(1-\theta)^2\nabla^3\varphi(\bar{X}^x_k+\theta\Delta \bar{X}^x_k)\big(\sigma(\bar{X}^x_k)\Delta W_k,\sigma(\bar{X}^x_k)\Delta W_k,\sigma(\bar{X}^x_k)\Delta W_k\big)\ud \theta.
\end{align*}
Similar to the derivation of \eqref{Eq4}, one can use \eqref{varphidiff}, \eqref{eq17} and Theorem \ref{pth-moment} to get that for any $p\ge1$ and $\tau<1$,
\begin{align}
	\mbf E\|\nabla^3\varphi(\bar{X}^x_k+\theta\Delta \bar{X}^x_k)\|^p_{\otimes}\le K(1+|x|^{2pq'q}),~\theta\in[0,1].\label{Eq6}
\end{align}
By \eqref{Eq5}, \eqref{Eq6} and the H\"older inequality, one has
\begin{align}
	\mbf E|\mcal R_{\tau,6}^1|\le K(x)\tau^{\frac32},\quad \mbf E|\mcal R_{\tau,6}^2|\le K(x)\tau,\quad \mbf E|\mcal R_{\tau,6}^3|\le K(x)\tau^{\frac12}.\label{Eq7}
\end{align}

Further, applying the Taylor expansion for $\nabla^3\varphi$,  we write $\mcal R_{\tau,6}^4=\mcal R_{\tau,6}^{4,1}+\mcal R_{\tau,6}^{4,2}$, where
\begin{align*}
	\mcal R_{\tau,6}^{4,1}&:=-\frac{\tau^{\frac12}}{6}\sum_{k=0}^{m-1}\nabla^3\varphi(\bar{X}^x_k)\big(\sigma(\bar{X}^x_k)\Delta W_k,\sigma(\bar{X}^x_k)\Delta W_k,\sigma(\bar{X}^x_k)\Delta W_k\big),\\
\mcal R_{\tau,6}^{4,2}&:=-\frac{\tau^{\frac12}}{2}\sum_{k=0}^{m-1}\int_0^1\int_0^1\nabla^4\varphi(\bar{X}^x_k+r\theta\Delta \bar X^x_k)\big(\sigma(\bar{X}^x_k)\Delta W_k,\sigma(\bar{X}^x_k)\Delta W_k,\sigma(\bar{X}^x_k)\Delta W_k,\Delta \bar X^x_k\big)\ud r \theta(1-\theta)^2\ud \theta.	
\end{align*} 
Similar to the proof for \eqref{Eq6}, we have that for any $p\ge1$ and $\tau<1$,
\begin{align*}
	\sup_{k\ge0}\mbf E\|\nabla^4\varphi(\bar{X}^x_k+r\theta\Delta \bar{X}^x_k)\|^p_{\otimes}\le K(1+|x|^{3pq'q}),~r,\theta\in[0,1].
\end{align*}
This together with the H\"older inequality and \eqref{eq17} gives
\begin{align*}
	\mbf E|\mcal R_{\tau,6}^{4,2}|\le K\tau^{\frac12}m\sup_{k\ge0}\Big[\sup_{r,\theta\in[0,1]}\big(\mbf E\|\nabla^4\varphi(\bar{X}^x_k+r\theta\Delta \bar{X}^x_k)\|^3_{\otimes}\big)^{\frac13}\big(\mbf E|\Delta W_k|^9\big)^{\frac13}\big(\mbf E|\Delta \bar X_k^x|^3\big)^{\frac13}\Big]\le K(x)\tau^{\frac12}.
\end{align*} 
Since $\Delta W_j$ is $\mcal F_{t_j}$-independent, for any $i<j$,
\begin{align*}
&\;\mbf E\big[\nabla^3\varphi(\bar{X}^x_i)\big(\sigma(\bar{X}^x_i)\Delta W_i,\sigma(\bar{X}^x_i)\Delta W_i,\sigma(\bar{X}^x_i)\Delta W_i\big)\nabla^3\varphi(\bar{X}^x_j)\big(\sigma(\bar{X}^x_j)\Delta W_j,\sigma(\bar{X}^x_j)\Delta W_j,\sigma(\bar{X}^x_j)\Delta W_j\big)\big]	\\
=&\;\mbf E\big[\nabla^3\varphi(\bar{X}^x_i)\big(\sigma(\bar{X}^x_i)\Delta W_i,\sigma(\bar{X}^x_i)\Delta W_i,\sigma(\bar{X}^x_i)\Delta W_i\big)\mbf E_j\big[\nabla^3\varphi(\bar{X}^x_j)\big(\sigma(\bar{X}^x_j)\Delta W_j,\sigma(\bar{X}^x_j)\Delta W_j,\sigma(\bar{X}^x_j)\Delta W_j\big)\big]\big]\\
=&\;0,
\end{align*}
where we used  the property of conditional expectations and
$$\mbf E(\Delta W_j^{p_1}\Delta W_j^{p_2}\Delta W_j^{p_3})=0\quad\forall~ p_1,p_2,p_3\in\{1,2,\ldots,D\},$$
with $\Delta W_j^{r}$ being the $r$th component of $\Delta W_j$. In this way, we get
\begin{align*}
	\mbf E|\mcal R_{\tau,6}^{4,1}|^2=\frac{\tau}{36}\sum_{k=0}^{m-1}\mbf E\Big[\nabla^3\varphi(\bar{X}^x_k)\big(\sigma(\bar{X}^x_k)\Delta W_k,\sigma(\bar{X}^x_k)\Delta W_k,\sigma(\bar{X}^x_k)\Delta W_k\big)\Big]^2\le K(x)\tau^2,
\end{align*}
due to \eqref{Eq4}. Thus, it holds that $\mbf E|\mcal R_{\tau,6}^4|\le K(x)\tau^{\frac12}$ for $\tau<1$, which combined with \eqref{Eq7} yields $\mbf E|\mcal R_{\tau,6}|\le K(x)\tau^{\frac12}$.

Combining the above estimates for $\mcal R_{\tau,i}$, $i=1,\ldots,6$, we obtain $\lim\limits_{\tau\to0}\mbf E|\mcal R_\tau|=0$. This gives the desired conclusion. \hfill$\square$ 

\section{Numerical experiments}\label{Sec5}
In this section, we perform numerical experiments to verify our theoretical results. First, for a given test function $h$, we obtain the approximation of the ergodic limit $\pi(h)$ numerically by virtue of the fact $\lim\limits_{t\to\infty}\mbf E(h(X(t)))=\pi(h)$ (see \eqref{pif}). Here, $\lim\limits_{t\to\infty}\mbf E(h(X(t)))$ is simulated by the numerical solution $\{\bar{X}_n\}_{n\ge0}$ of the BEM method. More precisely, let the step-size $\tau$ be small enough, $N$ sufficiently large, and use the  Monte–Carlo method to simulate the expectation. Then we have
$\lim\limits_{t\to\infty}\mbf E(h(X(t)))\approx \frac{1}{M}\sum\limits_{i=1}^Mh(\bar X_N^i)$,  with $\{\bar X_N^i\}_{i=1}^M$ being $M$ samplings of $\bar X_N$. Second, we verify the CLT for $\Pi_{\tau,\alpha}$, $\alpha\in(1,2]$. Denote $Z_{\tau,\alpha}(h)=	\frac{1}{\tau^{\frac{\alpha-1}{2}}}\Big(\frac{1}{\tau^{-\alpha}}\sum\limits_{k=0}^{\tau^{-\alpha}-1}h(\bar{X}_k)-\pi(h)\Big).$ Then, the CLT shows that for any $f\in\mbf C_b(\mbb R^d)$, $\lim\limits_{\tau\to0}\mbf Ef(Z_{\tau,\alpha}(h))=\int_{\mbb R^d}f(x)\mcal N(0,\pi(|\sigma^\top\nabla\varphi|^2))(\ud x)$. We will numerically verify that $\mbf Ef(Z_{\tau,\alpha}(h))$ tends to some constant as $\tau$ decreases.

\textbf{Example 5.1.} Consider the following SODE with Lipschitz diffusion coefficient:
\begin{align*}
	\begin{cases}
		\ud X(t)=-(X^3(t)+8X(t))\ud t+\sin(X(t))\ud W(t),\\
		X(0)=x\in\mbb R.	
	\end{cases}
\end{align*}
It is not difficult to verify that the coefficients of the above equation satisfy Assumptions \ref{assum1}-\ref{assum3}. First, we numerically simulate  the ergodic limit $\pi(h)$ using the aforementioned method. The expectation is realized by $5000$ sample paths. Fig.\ \ref{F1} displays the evolution of $\mbf Eh(\bar X_n)$ w.r.t. $n$ starting from different initial values. It is observed that the ergodic limit are $1$ and $0$ for  $h=\sin(x)+1$ and $h=x^4$, respectively.

\begin{figure}[H]
	\centering
	\subfigure[$h=\sin(x)+1$]{\includegraphics[width=0.49\textwidth]{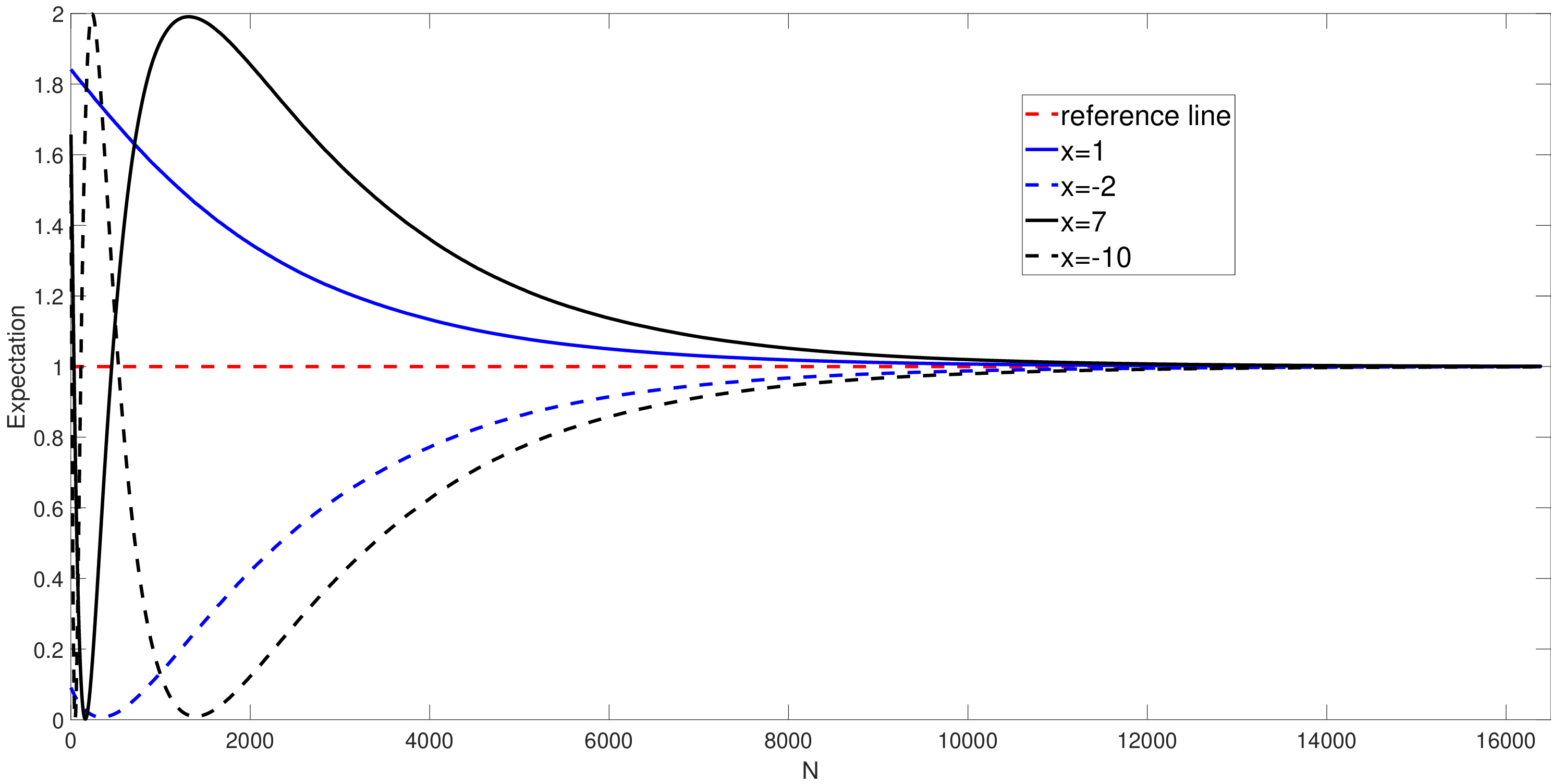}}
	\subfigure[$h=x^4$]{\includegraphics[width=0.49\textwidth]{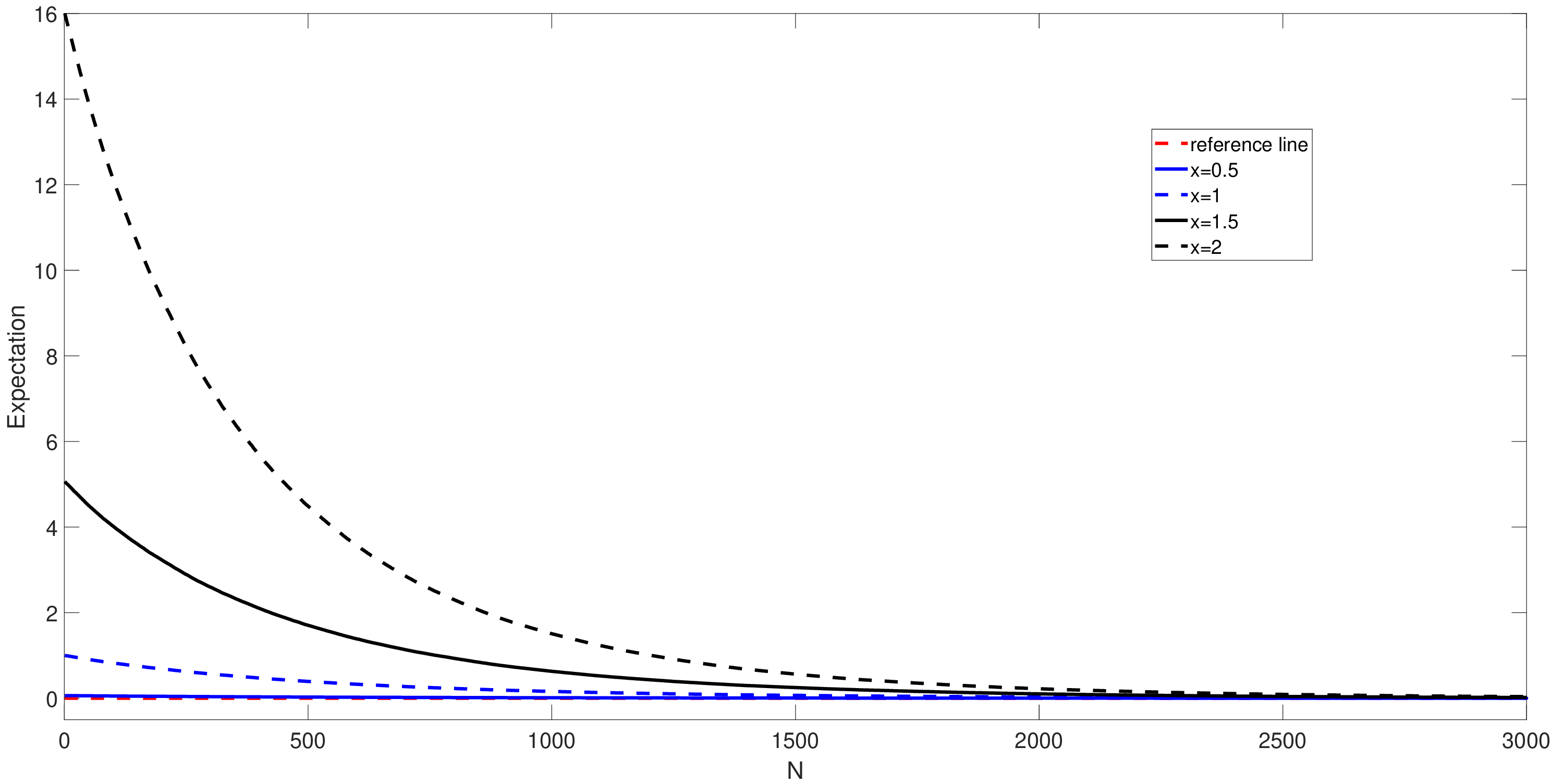}}
	\caption{The expectations of $h(\bar X_n)$ starting from different initial values with $\tau=2^{-14}$, $M=5000$.}\label{F1}
\end{figure}

Tables \ref{T1}-\ref{T4} show the evolution of $\mbf Ef(Z_{\tau,2}(h))$ w.r.t. $\tau$, where the initial value $x=1$ for Tables \ref{T1}-\ref{T2} while $x=-2$ for Tables \ref{T3}-\ref{T4}. It is observed that for all kinds of cases, $\mbf Ef(Z_{\tau,2}(h))$ will tend to some constant as $\tau$ decreases. We also find that the CLT of $\Pi_{\tau,2}$  also holds for  $h$ of super-linear growth. See also Section \ref{Sec6} for the discussion about this problem. 

\begin{table}[H]
	\centering
	\caption{The values of $\mbf Ef(Z_{\tau,2}(h))$ for different step-size with $x=1$, $M=5000$, $h(x)=\sin(x)+1.$}\label{T1}
		\vspace{2mm}
	\setlength{\tabcolsep}{0.5mm}{
		\begin{tabular}{| c| c| c| c| c| c|} 
			\hline
			$\tau$	& 0.05& 0.045 &0.04 & 0.035& 0.03\\ 
			\hline
			$f(x)=\cos(x)$
			& 0.9993475    & 0.9994458                    & 0.9995374   & 0.9996222                   & 0.9996928                 \\ 
			\hline
			$f(x)=\sin(x^6)$
			& 3.503257E-9   &2.157256E-9                     & 1.290072E-9 &   6.935839E-10                   &  3.801383E-10   \\ 
			\hline
			$\tau$	& 0.025&0.02&0.015&0.01&0.005  \\ 
			\hline
			$f(x)=\cos(x)$
			& 0.9997601                   & 0.9998199     & 0.999876                  & 0.9999221     & 0.9999635             \\ 
			\hline
			$f(x)=\sin(x^6)$
			&  1.787545E-10                & 7.860379E-11 &        2.516742E-11         & 6.635954E-12 &        7.200134E-13   \\ 
			\hline
	\end{tabular}}
\end{table}

\begin{table}[H]
	\centering
	\caption{The values of $\mbf Ef(Z_{\tau,2}(h))$ for different step-size with $x=1$, $M=5000$, $h(x)=x^4.$}\label{T2}
		\vspace{2mm}
	\setlength{\tabcolsep}{0.5mm}{
		\begin{tabular}{| c| c| c| c| c| c|} 
			\hline
		$\tau$	& 0.05& 0.045 &0.04 & 0.035& 0.03\\ 
			\hline
			$f(x)=\cos(x)$
		& 0.9998683    & 0.9998932                    & 0.9999163    & 0.9999356                    & 0.9999507                  \\ 
			\hline
	$f(x)=\sin(x^6)$
	&3.618809E-11    &2.394611E-11                     & 1.248624E-11 &   6.192004E-12                    &  3.409618E-12    \\ 
			\hline
			$\tau$	& 0.025&0.02&0.015&0.01&0.005  \\ 
			\hline
			$f(x)=\cos(x)$
			& 0.9999641                   & 0.9999746     & 0.9999842                   & 0.9999907     & 0.9999960             \\ 
			\hline
			$f(x)=\sin(x^6)$
			&  1.698598E-12                & 1.207866E-12 &        1.642608E-13         & 5.592859E-14 &        4.168847E-15   \\ 
			\hline
	\end{tabular}}
\end{table}

\begin{table}[H]
	\centering
	\caption{The values of $\mbf Ef(Z_{\tau,2}(h))$ for different step-size with $x=-2$, $M=5000$, $h(x)=\sin (x)+1.$}\label{T3}
		\vspace{2mm}
	\setlength{\tabcolsep}{0.5mm}{
		\begin{tabular}{| c| c| c| c| c| c|} 
			\hline
			$\tau$	& 0.05& 0.045 &0.04 & 0.035& 0.03\\ 
			\hline
			$f(x)=\cos(x)$
			& 0.998501    & 0.9987222                    & 0.9989255   & 0.9991153                    & 0.9992799                 \\ 
			\hline
			$f(x)=\sin(x^6)$
			&3.576542E-8    & 2.212666E-8                     & 1.329402E-8 &   7.365759E-9                    &  4.040548E-9    \\ 
			\hline
			$\tau$	& 0.025&0.02&0.015&0.01&0.005  \\ 
			\hline
			$f(x)=\cos(x)$
			& 0.999434                   & 0.9995741     & 0.9997039                  & 0.999814    & 0.9999128             \\ 
			\hline
			$f(x)=\sin(x^6)$
			&   1.944833E-9               & 8.400904E-10 &        2.824339E-10        & 7.212851E-11 &        7.596859E-12   \\ 
			\hline
	\end{tabular}}
\end{table}

\begin{table}[H]
	\centering
	\caption{The values of $\mbf Ef(Z_{\tau,2}(h))$ for different step-size with $x=-2$, $M=5000$, $h(x)=x^4.$}\label{T4}
		\vspace{2mm}
	\setlength{\tabcolsep}{0.5mm}{
		\begin{tabular}{| c| c| c| c| c| c|} 
			\hline
			$\tau$	& 0.05& 0.045 &0.04 & 0.035& 0.03\\ 
			\hline
			$f(x)=\cos(x)$
			& 0.9732673    & 0.9789388                    & 0.9838831 & 0.9879194                    & 0.9911355                \\ 
			\hline
			$f(x)=\sin(x^6)$
			&1.699370E-4   & 8.439149E-5                     & 3.866177E-5 &   1.664767E-5                    &  6.780718E-6    \\ 
			\hline
			$\tau$	& 0.025&0.02&0.015&0.01&0.005  \\ 
			\hline
			$f(x)=\cos(x)$
			& 0.9937917                   & 0.9958541    & 0.9974848                & 0.9986377    & 0.9994581             \\ 
			\hline
			$f(x)=\sin(x^6)$
			&   2.423985E-6               & 7.681994E-7 &        1.759620E-7        & 3.114619E-8 &        2.156873eE-9   \\ 
			\hline
	\end{tabular}}
\end{table}

\textbf{Example 5.2.} Consider the following SODE with non-Lipschitz diffusion coefficient:
\begin{align*}
	\begin{cases}
		\ud X(t)=-(X^3(t)+10X(t))\ud t+0.5X^2(t)\ud W(t),\\
		X(0)=x\in\mbb R.	
	\end{cases}
\end{align*}
Notice that the above equation satisfies Assumptions $2.1$-$2.4$ of \cite{Liuwei23}. Thus, 
$\{X(t)\}_{t\ge0}$ admits a unique invariant measure $\pi$.
Fig.\ \ref{F2} displays the evolution of $\mbf Eh(\bar X_n)$ w.r.t. $n$ starting from different initial values. In this case, the numerical ergodic limit is $0$. Table \ref{T5} reflects the evolution of $\mbf Ef(Z_{\tau,\alpha})$ as $\tau$ decreases. It is observed in Table \ref{T5} that $\mbf Ef(Z_{\tau,\alpha})$ will tend to $0$ for three different parameters $\alpha=1.2,1.5,2$. We remark that  the CLT may still hold for the BEM method of SODEs with non-Lipschitz diffusion coefficients, as is numerically shown in this example.

\begin{figure}[H]
	\centering
\includegraphics[width=0.9\textwidth]{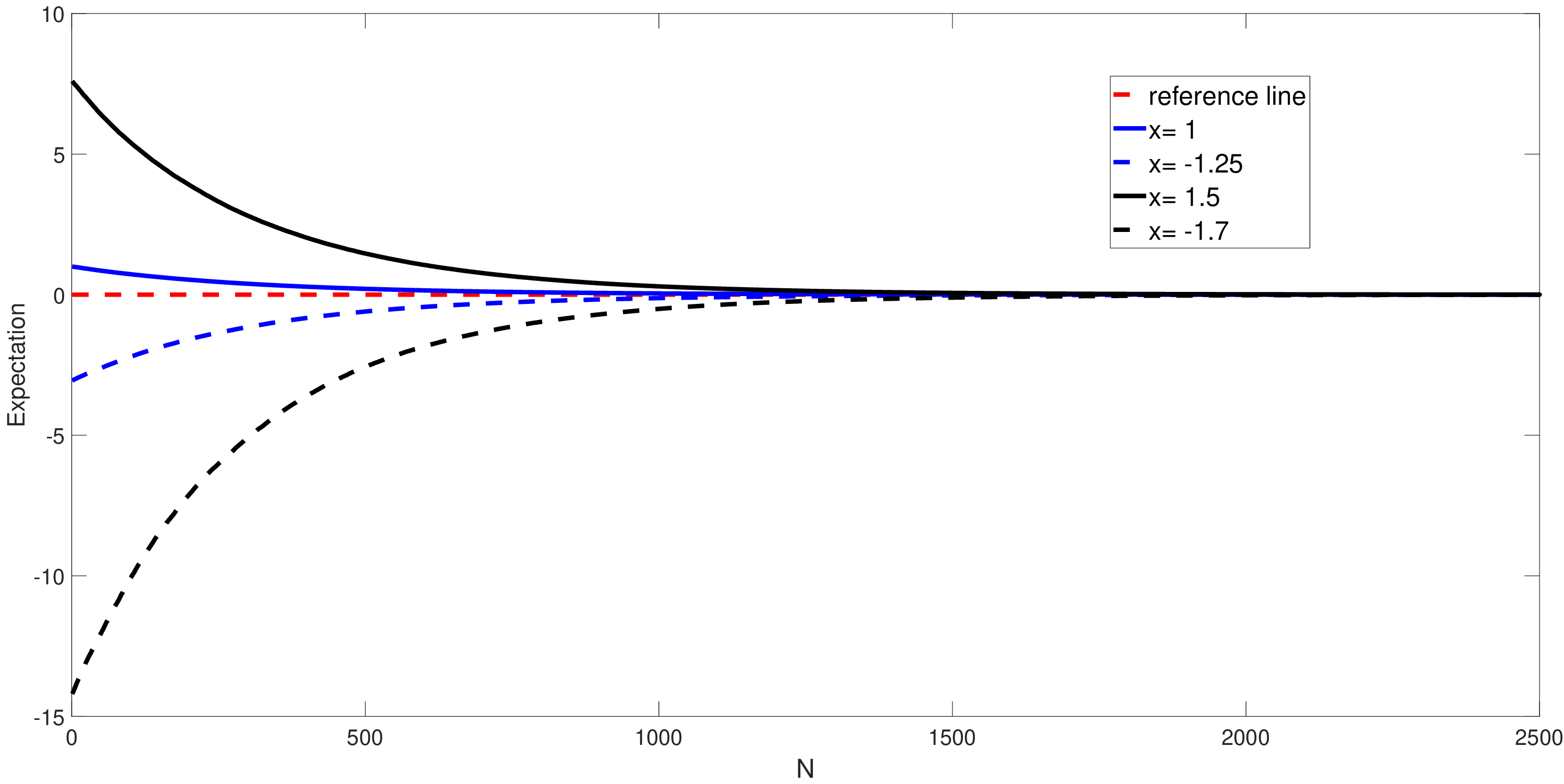}
	\caption{The expectations of $(\bar X_n)^5$ starting from different initial values with $\tau=2^{-14}$.}\label{F2}
\end{figure}

\begin{table}[H]
	\centering
	\caption{The values of $\mbf Ef(Z_{\tau,\alpha}(h))$ for different step-size with $x=0.5$, $M=5000$, $f(x)=\sin(x^6)$ and $h(x)=x^5.$}\label{T5}
	\vspace{2mm}
	\setlength{\tabcolsep}{0.5mm}{
		\begin{tabular}{| c| c| c| c| c| c|} 
			\hline
			$\tau$	& 0.01& 0.0095 &0.009 & 0.0085& 0.008\\ 
			\hline
			$\alpha=1.2$ 
			 &2.207120E-20 &2.008384E-20 &1.783228E-20 &1.614639E-20 &1.442169E-20          \\ 
			\hline
			$\alpha=1.5$
			 &3.482317E-22 &3.052800E-22 &2.552166E-22 &2.182177E-22 &1.865925E-22               \\ 
			\hline
			$\alpha=2$ 
			&3.482317E-25 &2.832017E-25 &2.203657E-25 &1.69844E-25 &1.339561E-25           \\ 
			\hline
			$\tau$	& 0.0075&0.007&0.0065&0.006&0.0055  \\ 
			\hline
			$\alpha=1.2$ 
			&1.32067E-20 &1.126548E-20 &1.009739E-20 &8.871488E-21 &7.774097E-21           \\ 
			\hline
			$\alpha=1.5$
			 &1.533266E-22 &1.296991E-22 &1.053313E-22 &8.859776E-23 &7.185399E-23           \\ 
			\hline
			$\alpha=2$
			&1.002341E-25 &7.552484E-26 &5.575008E-26 &4.003481E-26 &2.907394E-26          \\ 
			\hline
	\end{tabular}}
\end{table}

\section{Conclusions and future}\label{Sec6}
In this work, we prove the CLT for the temporal average of the BEM method, which characterizes the asymptotics of the BEM method in distribution.
The drift coefficients of underlying SODEs are allowed to grow super-linearly.  Different proof strategies are used for different deviation orders, which relies on the relationship  between the deviation order and optimal strong order of the BEM method.

In fact, it is possible to weaken the conditions of Theorems \ref{CLT1}-\ref{CLT2}, and we refer to  the following two aspects.
\begin{itemize}
	\item \textit{Conditions on $h$}. By revisiting the whole proof of Theorem \ref{CLT2}, it is observed that the requirement for the test function $h$ can be lowered. If we let $\nabla ^ih\in Poly(q'',\mbb R^d)$, $i=0,1,\ldots,4$ instead of $h\in \mbf C^4_b(\mbb R^d)$, then the main difference lies in the regularity of $\varphi$. In fact, it holds that $\nabla ^i\varphi \in Poly(L_0,\mbb R^d)$, $i=0,1,\ldots,4$ for some integer $L_0$ dependent on $q',q''$.  And this will make no difference to the conclusions of Lemmas \ref{Htaulimit} and \ref{Rtaulimit}, in view of Theorem \ref{pth-moment}. Thus, the CLT still holds for $\Pi_{\tau,2}(h)$ for a class of unbounded $h$. Similarly, Theorem \ref{CLT1} also holds for $\nabla^i h\in Poly(q'',\mbb R^d)$, $i=0,1,\ldots,4$. The above facts
	are also observed in the numerical experiments in Section \ref{Sec5}.
	
\item \textit{Conditions on $\sigma$}. Assume that $\sigma$ is unbounded but globally Lipschitz.  Let Assumption \ref{assum2} hold with $c_1>\frac{15}{2}L_1^2$ replaced by $c_1$ being sufficiently large. We can follow the same  argument in Theorem \ref{pth-moment} to give the $p$th moment boundedness for the BEM method. Roughly speaking, in this case, \eqref{Eq11} still holds. Similar to \eqref{eq2}, we obtain
\begin{align*}
	(1+pc_1\tau)|\bar{X}^x_{n+1}|^{2p}\le \big(|\bar{X}^x_{n}|^2+2\LL\bar{X}^x_{n},\sigma(\bar{X}^x_{n})\Delta W_n \RR+K(\tau+|\Delta W_n|^2)+K|\bar X^x_n|^2|\Delta W_n|^2\big)^p
\end{align*}
due to the linear growth of $\sigma$. By the similar analysis for \eqref{Eq12}, one can show that
\begin{align*}
\mbf E|\bar{X}^x_{n+1}|^{2p}\le \frac{(1+A(p,D)\tau)}{(1+pc_1\tau)}\mbf E|\bar X^x_n|^{2p}+K(p)\frac{(1+|x|^{2p-2})\tau}{(1+pc_1\tau)}
\end{align*}
for some $A(p,D)>0$ dependent on $p$ and $D$. Using the condition that $c_1$ is sufficiently large, one finally can obtain $\sup\limits_{n\ge0}\mbf E|\bar{X}^x_n|^r\le K(1+|x|^r)$ for some $r$  large enough. Thus, other conclusions still hold on basis of the moment boundedness of $\{\bar{X}_n\}_{n\ge0}$. Finally, one can establish the CLT for $\Pi_{\tau,\alpha}(h)$ when   $\sigma$ is Lipschitz, provided that the dissipation parameter $c_1$ is sufficiently large. When $\sigma$ is Lipschitz or of super-linear growth, it is interesting to study how to prove the $p$th ($p>2$) moment boundedness of the BEM method in the infinite time horizon for a relatively small $c_1$. We will study this problem in the future.
\end{itemize}

%\section*{references}
\bibliographystyle{plain}
\bibliography{mybibfile}

\end{document}